\documentclass[]{interact}

\usepackage{epstopdf}
\usepackage[caption=false]{subfig}

\usepackage[numbers,sort&compress]{natbib}
\bibpunct[, ]{[}{]}{,}{n}{,}{,}
\renewcommand\bibfont{\fontsize{10}{12}\selectfont}
\makeatletter
\def\NAT@def@citea{\def\@citea{\NAT@separator}}
\makeatother

\theoremstyle{plain}
\newtheorem{theorem}{Theorem}[section]
\newtheorem{lemma}[theorem]{Lemma}

\newtheorem{proposition}[theorem]{Proposition}

\theoremstyle{definition}
\newtheorem{definition}[theorem]{Definition}
\newtheorem{example}[theorem]{Example}
\newtheorem{assumption}{Assumption}

\theoremstyle{remark}
\newtheorem{remark}{Remark}

\usepackage{mathtools}
\usepackage{thmtools}
\usepackage{enumitem}
\usepackage{graphicx}%
\usepackage{multirow}%
\usepackage{amsmath,amssymb,amsfonts}%
\usepackage{amsthm}%
\usepackage{mathrsfs}%
\usepackage[title]{appendix}%
\usepackage{xcolor}%
\usepackage{textcomp}%
\usepackage{manyfoot}%
\usepackage{booktabs}%
\usepackage{algorithm}%
\usepackage{algorithmicx}%
\usepackage{algpseudocode}%
\usepackage{program}%
\usepackage{listings}%
\usepackage{ulem}
\usepackage{subfig}

\begin{document}
\newcommand{\Real}{\mathbb{R}}
\newcommand{\Natural}{\mathbb{N}}
\newcommand{\Id}{\mathrm{Id}}

\newcommand{\parentheses}[1]{\left( #1 \right )}
\newcommand{\brackets}[1]{\left[ #1 \right]}
\newcommand{\braces}[1]{\left\{ #1 \right\}}

\newcommand{\norm}[1]{\left\| #1 \right\|}
\newcommand{\argmin}{\mathrm{arg\,min}}
\newcommand{\innerp}[2]{\left\langle #1, #2 \right\rangle}
\newcommand{\lev}[1]{\mathrm{lev}_{\leq #1}}
\newcommand{\dom}{\mathrm{dom\,}}
\newcommand{\dist}{\mathrm{dist}}
\newcommand{\convhull}{\mathrm{Conv}}
\newcommand{\Prox}{\mathrm{Prox}}
\DeclarePairedDelimiter\abs{\lvert}{\rvert}
\newcommand{\txtred}[1]{\textcolor{red}{#1}}
\newcommand{\revise}[1]{\textcolor{red}{#1}}
\newcommand{\delete}[1]{\textcolor{green}{\sout{#1}}}

\articletype{}

\title{An Inexact Proximal Linearized DC Algorithm with Provably Terminating Inner Loop}

\author{
\name{Yi Zhang\thanks{Yi Zhang. Email: yizhang@sp.ict.e.titech.ac.jp} and Isao Yamada\thanks{Isao Yamada. Email: isao@sp.ict.e.titech.ac.jp}}
\affil{Department of Information and Communications Engineering, Tokyo Institute of Technology, Tokyo, Japan}
}

\maketitle

\begin{abstract}
Standard approaches to difference-of-convex (DC) programs require exact solution to a convex subproblem at each iteration, which generally requires noiseless computation and infinite iterations of an inner iterative algorithm. To tackle these difficulties, inexact DC algorithms have been proposed, mostly by relaxing the convex subproblem to an approximate monotone inclusion problem. However, there is no guarantee that such relaxation can lead to a finitely terminating inner loop. In this paper, we point out the termination issue of existing inexact DC algorithms by presenting concrete counterexamples. Exploiting the notion of $\epsilon$-subdifferential, we propose a novel inexact proximal linearized DC algorithm termed tPLDCA. Despite permission to a great extent of inexactness in computation, tPLDCA enjoys the same convergence guarantees as exact DC algorithms. Most noticeably, the inner loop of tPLDCA is guaranteed to terminate in finite iterations as long as the inner iterative algorithm converges to a solution of the proximal point subproblem, which makes an essential difference from prior arts. In addition, by assuming the first convex component of the DC function to be pointwise maximum of finitely many convex smooth functions, we propose a computational friendly surrogate of $\epsilon$-subdifferential, whereby we develop a feasible implementation of tPLDCA. Numerical results demonstrate the effectiveness of the proposed implementation of tPLDCA.
\end{abstract}

\begin{keywords}
DC programming, finite termination, proximal linearized method, global optimization, nested algorithm
\end{keywords}

\section{Introduction}\label{sec1}
The standard difference-of-convex (DC) programs refer to nonconvex optimization problems of the following form:
\begin{equation}\label{eq:dcp}
(\mathrm{DCP})\quad\underset{x\in\Real^n}{\text{minimize}}\;\;f(x)\coloneqq g(x)-h(x),
\end{equation}
where $g,h:\Real^n\to\Real\cup\{+\infty\}$ are proper, lower semicontinuous and (possibly nonsmooth) convex functions. DC programs form a large class of nonconvex optimization problems encountered in real world \cite{holmberg1999,neumann2005,bagirov2006,khalaf2017}, and have been studied extensively in the mathematical programming community \cite{hiriart1985,tuy1987,pham1997,horst1999,
le2005,le2012,pang2017,le2018}.

A standard approach to solve (DCP) is the DC Algorithm (DCA) \cite{pham1988,pham1997,le2005}, which generates iteratively a sequence of trial points $(x_k)_{k\in\Natural}$ by solving the following convex subproblem,
\begin{equation}\label{eq:sp}
(\mathrm{SP})\quad x_{k+1}\in\arg\min_{x\in\Real^n}\; \tilde{f}_k(x)\coloneqq g(x)-\innerp{u_k}{x-x_k},
\end{equation}
where $u_k\in \partial h(x_k)$. The proximal linearized algorithm for minimization of DC functions (PLDCA) \cite{sun2003,moudafi2006,souza2016,an2017} is a natural extension of DCA, where (SP) is replaced by the following proximal point subproblem,
\begin{equation}\label{eq:pp}
(\mathrm{PP})\quad x_{k+1}\in\arg\min_{x\in\Real^n}\; \tilde{f}_k(x)\coloneqq g(x)-\innerp{u_k}{x-x_k}+\frac{1}{2\lambda_k}\norm{x-x_k}^2.
\end{equation}

It has been proved that under certain conditions, every cluster point $\bar{x}$ of $(x_k)_{k\in\Natural}$ generated by DCA and PLDCA is a critical point of $g-h$, i.e., $\partial g(\bar{x})\cap \partial h(\bar{x})\neq \emptyset$ \cite{pham1997,le2018,sun2003,moudafi2006}. However, such convergence guarantee requires exact solution to convex subproblems (SP) and (PP). Although in some cases, simpler subproblems can be obtained by rearranging the DC decomposition \cite{pham1998,wen2018}, in general (SP) and (PP) are solved by an inner iterative process. In practice, such inner process is supposed to halt after finitely many iterations, whereby truncation error is induced. Additionally, there are situations where $u_k\in\partial h(x_k)$ can only be computed inexactly. In consequence, the convergence of conventional DCA and PLDCA remains questionable in application.

The aforementioned reasons have motivated research on inexact schemes of DCA and PLDCA. In earlier research, Sun et al. \cite{sun2003} proposed an inexact PLDCA exploiting the concept of $\epsilon$-subdifferential, but their method requires computing the "approximate" proximity operator $(\Id+\lambda \partial_{\epsilon}g)^{-1}$ which is difficult in general. Recent studies \cite{souza2016,oliveira2019,lu2019} are mostly based on the natural relaxation of (SP) or (PP) to the following approximate monotone inclusion problem:
\begin{equation}\label{eq:ap}
\mathrm{(AP)}\quad\dist\parentheses{0,\partial \tilde{f}_k(x_{k+1})}\leq \epsilon_k,
\end{equation}
or its variants (see (\ref{eq:constraint:xik:souza}) in Algorithm \ref{alg:souza}), where $\epsilon_k>0$ is a threshold parameter. Souza et al. \cite{souza2016} proposed a verifiable stopping criterion for the inner iterative algorithm based on (AP), and established convergence of the inexact PLDCA under the criterion. However, there is no guarantee that their stopping criterion can be satisfied in finitely many inner iterations. Oliveira and Tcheou \cite{oliveira2019} proposed a bundle like algorithm as the inner iterative algorithm. The bundle like algorithm is guaranteed to solve (AP) in finite iterations, and convergence of the resulting inexact DC algorithm has been established. Nevertheless, the memory complexity of their bundle like algorithm increases linearly with the number of inner iterations, and their proof of finite inner loop termination does not generalize to other iterative algorithms. Lu et al. \cite{lu2019} proposed an inexact version of the recently proposed Enhanced DCA \cite{pang2017} based on (AP). But as \cite{souza2016}, they did not prove the finite termination of the inner loop in \cite{lu2019}.

In this work, we aim to propose an inexact PLDCA with general guarantees on the finite termination of the inner loop. Our contribution is as follows:
\begin{enumerate}
\item[(a)] We point out the termination issue of existing inexact DC algorithms \cite{souza2016,oliveira2019,lu2019}. More particularly, on one hand, we give an illustrative example to show that (AP) may serve as a disappointing relaxation of (SP) for every $\epsilon_k>0$. On the other hand, we take the inexact PLDCA in \cite{souza2016} as a representative, and present a more concrete example to prove that relaxing (PP) into a variant of (AP) (i.e., (\ref{eq:constraint:xik:souza}) in Algorithm \ref{alg:souza}) may lead to an endless inner loop, even if the inner iterative algorithm converges to a solution of (PP).

\item[(b)] Exploiting the notion of $\epsilon$-subdifferential \cite[Ch. XI]{hiriart1993}, we propose a novel inexact PLDCA termed tPLDCA, where "t" means terminating. Despite its permission to a great extent of inexactness in computation, tPLDCA enjoys the same convergence guarantee as exact DC algorithms \cite{pham1997,sun2003}. Most noticeably, for any inner iterative algorithm convergent to a solution of (PP), the inner loop of tPLDCA is guaranteed to terminate within finite inner iterations, which makes an essential difference from prior arts \cite{souza2016,oliveira2019,lu2019}.

\item[(c)] We analyse possible computational difficulties regarding calculation of $\epsilon$-subdifferential, and propose a feasible solution under an additional assumption in case such difficulties occur. More precisely, by assuming that $g$ is the pointwise maximum of finitely many convex smooth functions, we introduce a computational friendly surrogate of $\epsilon$-subdifferential, named $\epsilon$-strict subdifferential. Employing it in place of $\epsilon$-subdifferential, we obtain a computational friendly implementation of tPLDCA without losing all of its favourable properties.
\end{enumerate}

The remainder of this paper is organized as follows. In Sec. \ref{sec:preliminaries}, we introduce the notation and basic assumptions adopted in this work. In Sec. \ref{sec:priorarts}, we point out the termination issue of existing inexact DC algorithms by presenting counterexamples. In Sec. \ref{sec:conceptual}, we present the proposed tPLDCA, and prove its finite inner loop termination as well as convergence properties. In Sec. \ref{sec:subdifferential}, we introduce the concept of $\epsilon$-strict subdifferential and show its advantages over other subdifferential concepts. In Sec. \ref{sec:implementation}, we propose a computational friendly implementation of tPLDCA by employing $\epsilon$-strict subdifferential in place of $\epsilon$-subdifferential. Sec. \ref{sec:numerical_results} reports some numerical results demonstrating the effectiveness of the proposed implementation of tPLDCA. Sec. \ref{sec:conclusion} closes this paper with some concluding remarks.

\section{Preliminaries}\label{sec:preliminaries}
\subsection{Notation}
We use $\Natural,\Real,\Real_{++}$ to denote the sets of natural numbers, real numbers and positive numbers. $\innerp{\cdot}{\cdot}$ and $\norm{\cdot}$ denote respectively the inner product and Euclidean norm on $\Real^n$. For $x\in\Real^n$ and $r>0$, $B(x;r)$ denotes the open ball $\braces{y\in\Real^n\mid \norm{y-x}<r}$ and $\bar{B}(x;r)$ denotes its closure. The distance between vector $x\in\Real^n$ and a nonempty closed set $S\subset\Real^n$ is denoted by $\dist(x,S)\coloneqq \min_{y\in S}\norm{x-y}$. For $S\subset\Real^n$, $\convhull\, S$ stands for the convex hull of $S$ \cite[Def. 3.3]{bauschke2017}.

$\Gamma_0(\Real^n)$ denotes the set of all proper, lower semicontinuous and convex functions from $\Real^n$ to $\Real\cup\braces{+\infty}$. The subgradient, subdifferential, $\epsilon$-subgradient and $\epsilon$-subdifferential of a convex function $\phi$ at $x\in\Real^n$ are defined as follows.

\begin{definition}[Subgradient and subdifferential {\cite[pp. 214--215]{rockafellar2015}}]
Let $\phi\in\Gamma_0(\Real^n)$. Given $x\in\dom{\phi}$, the vector $s\in\Real^n$ is called a subgradient of $\phi$ at $x$ if
\begin{equation*}
\parentheses{\forall y\in\Real^n}\;\;\phi(y)\geq \phi(x)+\innerp{s}{y-x}.
\end{equation*}
The set of all subgradients of $\phi$ at $x$ is called the subdifferential of $\phi$ at $x$, and is denoted by $\partial \phi(x)$.
\end{definition}

\begin{definition}[$\epsilon$-subgradient and $\epsilon$-subdifferential {\cite[Ch. XI, Def. 1.1.1]{hiriart1993}}]\label{def:esubdiff}
Let $\phi\in\Gamma_0(\Real^n)$. Given $x\in\dom{\phi}$ and $\epsilon>0$, the vector $s\in\Real^n$ is called an $\epsilon$-subgradient of $\phi$ at $x$ if
\begin{equation*}
\parentheses{\forall y\in\Real^n}\;\;\phi(y)\geq \phi(x)+\innerp{s}{y-x}-\epsilon.
\end{equation*}
The set of all $\epsilon$-subgradients of $\phi$ at $x$ is called the $\epsilon$-subdifferential of $\phi$ at $x$, and is denoted by $\partial_{\epsilon}\phi(x)$.
\end{definition}

The following propositions list some useful properties of $\epsilon$-subdifferential.

\begin{proposition}[See {\cite[Ch. XI, Thm. 1.1.2 and 1.1.4]{hiriart1993}}]\label{pp:convexity:esubdiff}
Let $\phi\in\Gamma_0(\Real^n)$, $\epsilon> 0$. Then for every $x\in\dom{\phi}$, $\partial_{\epsilon} \phi(x)$ is a nonempty closed convex set.
\end{proposition}

\begin{proposition}[Transportation formula {\cite[Ch. XI, Prop. 4.2.2]{hiriart1993}}]\label{pp:transportation:formula}
Let $\phi\in\Gamma_0(\Real^n)$, $x,x'\in\dom{\phi}$, and $s'\in\partial \phi(x')$. Then, for $\epsilon>0$, $s'\in\partial_{\epsilon}\phi(x)$ if and only if
\begin{equation*}
\phi(x')\geq \phi(x)+\innerp{s'}{x'-x}-\epsilon.
\end{equation*}
\end{proposition}
It should be noted that the transportation formula above essentially indicates some important continuity property of $\epsilon$-subdifferential which is not shared by subdifferential (see Sec. \ref{sec:continuity:esubdiff} in the present paper for detail). We can have a glance at this distinguishing feature in the following example. 

\begin{example}\label{exp:discontinuity:subdifferential}
Consider the simplest 1D case where $\phi(\cdot)=\lvert \cdot \rvert\in\Gamma_0(\Real)$.
Let $\bar{z}=0$, $\bar{v}=0\in\partial \phi(\bar{z})$, and let $(z_i)_{i\in\Natural}$ be a sequence convergent to $\bar{z}$. If $z_i$ does not coincide with the limit point $\bar{z}$ for every $i\in\Natural$, then it can be verified that
\begin{equation*}
\lim_{i\to +\infty}\dist\parentheses{\bar{v},\partial \phi(z_i)}=1\neq 0= \dist\parentheses{\bar{v},\partial \phi(\bar{z})},
\end{equation*}
which implies that $\dist\parentheses{\bar{v},\partial \phi(z)}$ is discontinuous at $\bar{z}$.

In contrast, let $\epsilon>0$, then from \cite[p. 93]{hiriart1993},
\begin{equation*}
\partial_{\epsilon} \phi(z)=\begin{cases}
[-1,-1+\epsilon/z],& \text{if }z<-\epsilon/2,\\
[-1,+1], & \text{if }-\epsilon/2\leq z\leq \epsilon/2,\\
[1-\epsilon/z,1], & \text{if }z>\epsilon/2.
\end{cases}
\end{equation*}
Hence
\begin{equation*}
\lim_{i\to +\infty}\dist\parentheses{\bar{v},\partial_{\epsilon} \phi(z_i)}=0= \dist\parentheses{\bar{v},\partial_{\epsilon} \phi(\bar{z})},
\end{equation*}
which implies that $\dist\parentheses{\bar{v},\partial_{\epsilon} \phi(z)}$ is continuous at $\bar{z}$.
\end{example}

Keeping the above example in mind, one would verify later that the discontinuity of distance function to subdifferential essentially leads to the failure of inner loop termination of existing inexact DC algorithms (see Example \ref{exp:nontermination:illustrative} and \ref{exp:nontermination:souza}), whilst the "continuity" of that of $\epsilon$-subdifferential leads to the guaranteed inner loop termination of the proposed tPLDCA (see Proposition \ref{pp:constraint:xi:finite} and Theorem \ref{thm:finite:termination}).

\subsection{Problem Setup}
In this paper, we consider (DCP) in (\ref{eq:dcp}) with $g,h\in\Gamma_0(\Real^n)$ under the following assumptions.
\begin{assumption}\label{assumpt:inff}
$\inf_{x\in\Real^n}f(x)>-\infty$.
\end{assumption}
\begin{assumption}\label{assumpt:domh}
$\dom{h}\coloneqq\braces{x\in\Real^n\mid h(x)<+\infty}=\Real^n$.
\end{assumption}
\begin{assumption}\label{assumpt:domg}
$g=g_0+\iota_C$, where $g_0\in\Gamma_0(\Real^n)$ with $\dom{g_0}=\Real^n$, $C$ is a nonempty closed convex set and $\iota_C$ is the indicator function defined as follows,
\begin{equation*}
\iota_C(x)=\begin{cases}
0, & \textrm{if }x\in C,\\
+\infty, & \textrm{otherwise.}
\end{cases}
\end{equation*}
\end{assumption}
\begin{assumption}\label{assumpt:solution:steptwo}
There exists an inner iterative algorithm $T_g$ such that for any input parameters $(x_k,u_k,\lambda_k)\in\dom{g}\times\Real^n\times \Real_{++}$, $T_g(x_k,u_k,\lambda_k)$ generates a sequence of inner iterates $(z_{i})_{i\in\Natural}$ in $\dom{g}$ convergent to a solution of (PP):
\begin{equation*}
(\mathrm{PP})\quad\arg\min_{z\in\Real^n} \tilde{f}_k(z)\coloneqq g(z)-\innerp{u_k}{z-x_k}+\frac{1}{2\lambda_k}\norm{z-x_k}^2.
\end{equation*}
More precisely, there exists some $ \bar{z}_{k}\in\dom{g}$ such that $\lim_{i\to +\infty}z_i=\bar{z}_{k}$ and
\begin{equation}\label{eq:barz}
0\in \partial \tilde{f}_k(\bar{z}_{k})= \partial g(\bar{z}_{k})-u_k+\frac{1}{\lambda_k}(\bar{z}_{k}-x_k).
\end{equation}
\end{assumption}
\begin{remark}[Justification for Assumption \ref{assumpt:inff}, \ref{assumpt:domh}, \ref{assumpt:domg}, \ref{assumpt:solution:steptwo}]~
\begin{enumerate}
\item[1.] Since (DCP) is supposed to find a minimizer of $f$, it is reasonable to assume that $f$ is bounded from below. 

\item[2.] If there exists $ x_*\not\in\dom{h}$, then the value of $f(x_*)$ is either $-\infty$ or $(\infty-\infty)$, thus we assume $\dom{h}=\Real^n$ to avoid unnecessary ambiguity.

\item[3.] Assumption \ref{assumpt:domg} rules out nasty instances in $\Gamma_0(\Real^n)$ that have unclosed domains or have discontinuities on the boundary of their domains.

\item[4.] Assumption \ref{assumpt:solution:steptwo} ensures feasibility of finding approximate solutions of (PP) with arbitrarily high precision.
\end{enumerate}
\end{remark}

Under the aforementioned assumptions, the following hold for (DCP) in (\ref{eq:dcp}).

\begin{lemma}\label{lemma:domf}
Define $\dom{f}\coloneqq \braces{x\in\Real^n\mid f(x)<+\infty}$, then $\dom{f}=C$ is a closed convex set.
\end{lemma}
\begin{proof}
Since $h$ is proper and $\dom{h}=\Real^n$ (Assumption \ref{assumpt:domh}), we yield that $h(x)\in\Real$ for every $x\in\Real^n$. Hence
\[\dom{f}=\braces{x\in\Real^n\mid g(x)<h(x)+\infty=+\infty},\]
which implies that $\dom{f}=\dom{g}=C$ (Assumption \ref{assumpt:domg}) is closed and convex.
\end{proof}
\begin{lemma}\label{lemma:contghf}
$g$, $h$ and $f$ are continuous over $\dom{f}$.
\end{lemma}
\begin{proof}
Since $g_0,h\in\Gamma_0(\Real^n)$ and $\dom{g_0}=\dom{h}=\Real^n$, $g_0,h$ are continuous on $\Real^n$ from \cite[Coroll 8.40]{bauschke2017}. Since the restriction of $g,h,f$ to $\dom{f}=C$ are identical to the restriction of $g_0,h,g_0-h$ to $\dom{f}$ respectively, we conclude that $g,h,f$ are continuous relative to $\dom{f}$.
\end{proof}

\begin{lemma}\label{lemma:bounded:partialhS}
Suppose that $\epsilon>0$ and that $S\subset\Real^n$ is bounded, then $\partial_{\epsilon} h(S)\coloneqq \cup_{x\in S} \partial_{\epsilon} h(x)$ is bounded.
\end{lemma}
\begin{proof}
Boundedness of $S$ ensures the existence of some $R>0$ satisfying\[\parentheses{\forall x\in S}\;\; \bar{B}(x;1)\subset \bar{B}(0;R).\]
For every $s_*\in\partial_{\epsilon}h(S)\setminus\braces{0}$, there exists $x_*\in S$ such that $s_*\in\partial_{\epsilon}h(x_*)$. Take $y_*=x_*+ s_*/\norm{s_*}$, then $y_*\in \bar{B}(x_*;1)\subset \bar{B}(0;R)$. Since $\dom{h}=\Real^n$ and $\bar{B}(0;R)$ is closed and bounded, $h$ is Lipschitz continuous relative to $\bar{B}(0;R)$ from \cite[Coroll. 8.41]{bauschke2017}. Let $L$ be a Lipschitz constant for $h$ on $\bar{B}(0;R)$. From Lipschitz continuity of $h$ on $\bar{B}(0;R)$ and $s_*\in\partial_{\epsilon}h(x_*)$, we have
\begin{equation*}
h(x_*)+L\geq h(y_*)\geq h(x_*)+\innerp{s_*}{{s_*}/{\norm{s_*}}}-\epsilon,
\end{equation*}
which implies $\norm{s_*}\leq L+\epsilon$. Therefore, $\partial_{\epsilon}h(S)\subset \bar{B}(0;L+\epsilon)$, thus is bounded.
\end{proof}

\section{Termination Issue of Prior Arts}
\label{sec:priorarts}
As introduced earlier, conventional DC algorithms require solving (SP) in (\ref{eq:sp}) or (PP) in (\ref{eq:pp}) exactly, which is equivalent to solving the following monotone inclusion problem:
\begin{equation}\label{eq:PP:subdiffcond}
0\in\partial \tilde{f}_k(x_{k+1})
\iff \dist\parentheses{0,\partial \tilde{f}_k(x_{k+1})}=0.
\end{equation}
Thanks to the advancement of convex optimization studies, in many cases we can expect (\ref{eq:PP:subdiffcond}) to be solved by some iterative algorithm (cf. $T_g$ in Assumption \ref{assumpt:solution:steptwo}), in the sense that it generates a sequence of inner iterates $(z_i)_{i\in\Natural}$ convergent to a solution of (\ref{eq:PP:subdiffcond}). However, such convergence generally requires an infinite amount of iterations, which yields an endless inner loop. To promote termination of the inner loop and also to permit calculation error, a natural solution is to relax (\ref{eq:PP:subdiffcond}) into
\begin{equation*}
(\mathrm{AP})\quad\dist(0,\partial \tilde{f}_k(x_{k+1}))\leq \epsilon_k
\end{equation*}
or its variants (e.g., (\ref{eq:constraint:xik:souza}) in Algorithm \ref{alg:souza}), and expect that (AP) can be solved within finite iterations of $T_g$, where $\epsilon_k>0$ is a threshold parameter.

The preceding idea constitutes the basic strategy of existing inexact DC algorithms \cite{souza2016,oliveira2019,lu2019}. However, the intuition that (AP) can be solved in finite inner iterations relies on an implicit belief, that is, $\dist(0,\partial \tilde{f}_k(z_i))$ can become sufficiently small as long as the inner estimate $z_i$ gets close enough to a solution of (\ref{eq:PP:subdiffcond}). Unfortunately, if we recall Example \ref{exp:discontinuity:subdifferential}, we would realize that such belief does not always hold because of discontinuity of the distance function to subdifferential. In the sequel, we present counterexamples to support this claim.

We first give an illustrative example to show that inexact DCAs relaxing (SP) to (AP) (e.g., \cite{lu2019}) may face termination issues in the inner loop.

\begin{example}\label{exp:nontermination:illustrative}
Consider the simplest 1D case where $g(\cdot)\coloneqq \abs{\cdot}$. Suppose that $u_k=0$, then $\tilde{f}_k(x)=\abs{x}$ in (SP) in this case. It can be verified that for every $\epsilon_k\in (0,1)$, 
\begin{equation*}
\mathrm{(SP)}\iff(\ref{eq:PP:subdiffcond})\iff \mathrm{(AP)}\iff x_{k+1}=0,\end{equation*}
which means that in this case, the relaxed problem (AP) is as difficult as the exact version of (SP). Hence as long as (SP) cannot be solved within finite inner iterations, the same is (AP). Moreover, if $\epsilon_k\in[1,+\infty)$, it can be verified that the solution set of (AP) is $\Real$, which yields a trivial relaxation of (SP). Therefore, whatever value of the threshold parameter $\epsilon_k$ is adopted, (AP) does not serve as a favourable relaxation of (SP) for this simple example. 
\end{example}

Next, we take the inexact PLDCA in \cite{souza2016}  (see Algorithm \ref{alg:souza}) as a representative, and present a more concrete example to show that relaxing (PP) into a variant of (AP) (i.e., (\ref{eq:constraint:xik:souza}) in Algorithm \ref{alg:souza}) may lead to endless inner loop. We note that (9) in \cite{oliveira2019} follows the same relaxation strategy, thus our example also applies there.
\begin{algorithm}
\textbf{Parameters:} Set $\theta>0$, $\sigma\in[0,1)$ and a bounded sequence of positive numbers $(\lambda)_{k\in\Natural}$ such that $\liminf_{k\to +\infty} \lambda_k>0$.

\vspace{1em}
Set $k=0$, $x_0\in\dom{f}$. Repeat the following steps until convergence.
\vspace{1em}

\textbf{Step 1:} find $u_k$ such that
\begin{equation}\label{eq:constraint:uk:souza}
u_k\in\partial h(x_k).
\end{equation}

\vspace{0.5em}
\textbf{Step 2:} set $z_{-1}=x_k$ and generate $z_0,z_1,\dots$ by $T_g(x_k,u_k,\lambda_k)$ until
\begin{equation}\label{eq:decrease:g:souza}
g(x_{k})-g(z_i)-\innerp{u_{k}}{x_k-z_i}\geq \frac{(1-\sigma)}{2\lambda_k}\norm{z_i-x_k}^2
\end{equation}
$\quad\quad\quad\;\;\;$and
\begin{equation}\label{eq:constraint:xik:souza}
\dist(u_k,\partial g(z_i))\leq \theta \norm{z_i-x_k}
\end{equation}
$\quad\quad\quad\;\;\;$hold simultaneously for some $i\in\braces{-1}\cup\Natural$. Then set $x_{k+1}\coloneqq z_i$.

\textbf{Step 3:} $k\leftarrow k+1$.
\caption{Souza, Oliveira and Soubeyran's algorithm \cite{souza2016}}
\label{alg:souza}
\end{algorithm}

In Algorithm \ref{alg:souza}, the step 1 performs exact computation of a subgradient of $h$ at $x_k$, and the step 2 updates the new iterate $x_{k+1}$ with an approximate solution of (PP) satisfying (\ref{eq:decrease:g:souza}) and (\ref{eq:constraint:xik:souza}) \footnote{In order to conduct a fair comparison with the proposed method, here we set $z_{-1}=x_k$ and generate the other inner iterates by $T_g$, as the proposed tPLDCA (Algorithm \ref{alg:concept}) does.}. One can verify from Assumption \ref{assumpt:solution:steptwo} that the sequence of inner iterates $(z_i)_{i\in\Natural}$ converges to a solution of (PP).

Here we focus on the inner loop of step 2. Replacing $z_i$ by $x_{k+1}$, we note that (\ref{eq:decrease:g:souza}) is equivalent to
\begin{equation*}
\tilde{f}_k(x_{k+1})\leq \tilde{f}_k(x_k)+\frac{\sigma}{2\lambda_k}\norm{x_{k+1}-x_k}^2,
\end{equation*}
which serves as a descent condition for some auxiliary function involving $f$ (cf. Theorem \ref{thm:convergence:functional}(a)); see (\ref{eq:pp}) for the expression of $\tilde{f}_k$. On the other hand, one can verify that the condition (\ref{eq:constraint:xik:souza}) implies that
\begin{equation*}
\dist\parentheses{u_k+\frac{1}{\lambda_k}(x_{k+1}-x_k),\partial g(x_{k+1})}\leq \parentheses{\theta+\frac{1}{\lambda_k}}\norm{x_{k+1}-x_k}^2,
\end{equation*}
which is a equivalent to
\begin{equation*}
\dist\parentheses{0,\partial \tilde{f}_k(x_{k+1})}\leq \parentheses{\theta+\frac{1}{\lambda_k}}\norm{x_{k+1}-x_k}^2,
\end{equation*}
thus (\ref{eq:constraint:xik:souza}) is a variant of (AP). Therefore, to show that Algorithm \ref{alg:souza} may face a termination issue in the inner loop of step 2, it is sufficient to prove: \textit{whatever values of $\theta, \lambda_k$ are adopted, there exists $(x_k,u_k)\in\Real^{n}\times\Real^n$ and a sequence $(z_i)_{i\in\Natural}$ satisfying Assumption \ref{assumpt:solution:steptwo}, such that for the resulting problem instance, (\ref{eq:constraint:xik:souza}) cannot be satisfied within finite iterations of $(z_i)_{i\in\Natural}$}. Our counterexample is as follows.
\begin{example}\label{exp:nontermination:souza}
Consider the simplest 1D case where $g(\cdot)=\abs{\cdot}$, then for given hyperparameters $\theta>0$ and $\lambda_k>0$, we analyse the case where $u_k=0$, $x_k\in\Real$ is a fixed real number. Then it can be verified that (PP) has a unique solution with the following closed-form expression:
\begin{equation*}
\bar{z}_{k}=\Prox_{\lambda_k g}(x_k)=
\begin{cases}
0, & \text{if } \abs{x_k}\leq\lambda_k\\
\mathrm{sign}(x_k)(\abs{x_k}-\lambda_k), & \text{if }\abs{x_k}> \lambda_k
\end{cases}
\end{equation*}
where $\Prox_{\lambda_k g}$ is the proximity operator of $\lambda_k g$ \cite[Sec. 12.4]{bauschke2017}. 

Let $x_k=\min\{1/(2\theta),\lambda_k\}$, then $\bar{z}_{k}=\Prox_{\lambda_k g}(x_k)=0$. Let the sequence of inner estimates $(z_i)_{i\in\Natural}$ be as follows,
\begin{equation*}
z_i=\frac{x_k}{2^i},\;\; i=0,1,\dots ,
\end{equation*}
then $\lim_{i\to +\infty} z_i=0=\bar{z}_{k}$, which verifies that $(z_i)_{i\in\Natural}$ satisfies Assumption \ref{assumpt:solution:steptwo}.

Now let us check (\ref{eq:constraint:xik:souza}) for the problem instance where 
\[u_k=0,\;\; x_k=\min\{1/(2\theta),\lambda_k\}\;\; \textrm{and}\;\; z_i=x_k/(2^i)\;\; \textrm{for every}\;\; i\in\Natural.\]
Note that $z_{-1}=x_k=z_0$, we only need to check $z_i$ for $i\in\Natural$. Since $z_i>0$ for every $i\in\Natural$, we have $\partial g(z_i)\equiv \braces{+1}$, which implies the following:
\begin{equation*}
(\forall i\in\Natural)\;\;\dist(u_k,\partial g(z_{i}))\equiv 1>\frac{1}{2}\geq \parentheses{1-\frac{1}{2^{i}}}\theta x_k=\theta\norm{z_{i}-x_k}.
\end{equation*}
Therefore, (\ref{eq:constraint:xik:souza}) cannot be satisfied within finite iterations of $(z_i)_{i\in\Natural}$.
\end{example}

Example \ref{exp:nontermination:illustrative} and \ref{exp:nontermination:souza} reveal a critical limitation of prior studies \cite{souza2016,oliveira2019,lu2019}, that is, \textit{even with a convergent inner iterative algorithm, there is no proof that the naive relaxation (AP) leads to a finitely terminating inner loop}. (Remark: Oliveira and Tcheou \cite{oliveira2019} proved the inner loop termination of their inertial DC algorithm if a specific bundle like algorithm is employed as the inner algorithm. However, their proof does not generalize to other inner iterative algorithms.)

Fortunately, this termination issue can be resolved by the proposed tPLDCA that follows, via exploiting the "continuity" property of $\epsilon$-strict subdifferential.

\section{The Proposed Inexact PLDCA with Terminating Inner Loop}
\label{sec:conceptual}
\begin{algorithm}
\textbf{Parameters:} Set $\sigma\in(0,1),\lambda>0,\theta>\frac{1}{\lambda},\rho\in\Big[0,\frac{1-\sigma}{\lambda}\Big),\gamma\in\Big[0,\frac{(1-\sigma)}{\lambda}-\rho\Big)$, $\lambda_k\coloneqq \lambda$ for every $k\in\Natural$ and set a sequence of positive numbers $(\zeta_k)_{k\in\Natural}$ such that $\lim_{k\to +\infty} \zeta_k=0$.

\vspace{1em}
Set $k=0$, $x_{-1}, x_0\in\dom{f}$. Repeat the following steps until convergence.
\vspace{1em}

\textbf{Step 1:} find $u_k\in\Real^n$ such that
\begin{equation}\label{eq:constraint:uk}
\dist(u_k,\partial_{\epsilon_k}h(x_k))\leq d_k,
\end{equation}
$\quad\quad\quad\;\;\;$where $\epsilon_k=\rho\norm{x_k-x_{k-1}}^2$, $d_k=\gamma\norm{x_k-x_{k-1}}$.

\textbf{Step 2:} set $z_{-1}=x_k$ and generate $z_0,z_1,\dots$ by $T_g(x_k,u_k,\lambda_k)$ until
\begin{equation}\label{eq:decrease:g}
g(x_k)-g(z_{i})-\innerp{u_k}{x_k-z_{i}}\geq \frac{(1-\sigma)}{\lambda_k}\norm{z_{i}-x_k}^2
\end{equation}
$\quad\quad\quad\;\;\;$and
\begin{equation}\label{eq:constraint:xik}
\dist(u_k,\partial_{\zeta_k}g(z_{i}))\leq \theta\norm{z_{i}-x_k}.
\end{equation}
$\quad\quad\quad\;\;\;$hold simultaneously for some $i\in\braces{-1}\cup\Natural$. Then set $x_{k+1}\coloneqq z_i$.

\textbf{Step 3:} $k\leftarrow k+1$.
\caption{Inexact PLDCA with Terminating Inner Loop (tPLDCA)}
\label{alg:concept}
\end{algorithm}

In this section, we propose an inexact PLDCA with terminating inner loop and term this algorithm tPLDCA ("t" means terminating). The proposed tPLDCA (see Algorithm \ref{alg:concept}) serves as an extension of Algorithm \ref{alg:souza}, where the improvements we made are as follows:
\begin{enumerate}[label=(\alph*)]
\item In the step 1, we relax (\ref{eq:constraint:uk:souza}) into (\ref{eq:constraint:uk}) to permit calculation error in the computation of subgradients.

\item In the step 2, we relax $\partial g(z_i)$ in (\ref{eq:constraint:xik:souza}) into $\partial_{\zeta_k} g(z_i)$ in (\ref{eq:constraint:xik}), whereby we obtain a looser constraint for $x_{k+1}$.
\end{enumerate}
With respect to the advantages of the improvements above, we note:
\begin{enumerate}[label=(\alph*)]
\item In some previous studies \cite{sun2003}, $u_k\in\partial_{\epsilon_k}h(x_k)$ has been adopted in the step 1 to permit inexactness. However, since $\partial_{\epsilon_k}h(x_k)$ does not necessarily cover a neighbourhood of $\partial h(x_k)$ (e.g., consider the case where $h$ is affine), the prior arts only permit calculation error to a very limited extent. In contrast, it can be verified that the solution set of (\ref{eq:constraint:uk}) always covers a neighbourhood of $\partial h(x_k)$, which indeed improves the robustness of the DC algorithm. 

\item As will be shown in Sec. \ref{subsec:finite:terminate:CiPLDCA}, the relaxation (\ref{eq:constraint:xik}) actually ensures finite termination of $T_g(x_k,u_k,\lambda_k)$ in the step 2, which makes an essential difference from existing inexact DC algorithms \cite{souza2016,oliveira2019,lu2019}.
\end{enumerate}
In the sequel, we present favourable properties of the proposed tPLDCA, i.e., finite termination of the inner loop and convergence properties.

\subsection{Finite Termination of the Inner Loop}
\label{subsec:finite:terminate:CiPLDCA}
In this section, we prove the most distinguished property of tPLDCA, that is: \textit{the step 2 of tPLDCA provably terminates in finite inner iterations, whatever algorithm $T_g$ (satisfying Assumption \ref{assumpt:solution:steptwo}) is adopted}.

Before proving this claim, we introduce a certain important continuity property of $\epsilon$-subdifferential which accounts for the finite inner loop termination.

\subsubsection{A Continuity Property of $\epsilon$-Subdifferential}
\label{sec:continuity:esubdiff}
Recall Example \ref{exp:discontinuity:subdifferential}, \ref{exp:nontermination:illustrative} and \ref{exp:nontermination:souza}, one may realize that the failure of (\ref{eq:constraint:xik:souza}) to be satisfied in finite iterations originates from the discontinuity of the distance function to subdifferential. More particularly, even if the inner iterate $z_i$ converges to $\bar{z}_{k}$ (see Assumption \ref{assumpt:solution:steptwo} for the definition), $\lim_{i\to +\infty}\dist\parentheses{u_k,\partial g(z_i)}$ may deviate from $\dist\parentheses{u_k,\partial g(\bar{z}_{k})}$. In contrast, $\epsilon$-subdifferential does not suffer from such discontinuity. This property of $\epsilon$-subdifferential is formalized as follows.
\begin{proposition}\label{pp:continuity:dist:esubdiff}
Suppose that $\zeta>0$, $\bar{z}\in\dom{g}$, $\bar{v}\in\partial g(\bar{z})$. Then for any sequence $(z_i)_{i\in\Natural}$ in $\dom{g}$ convergent to $\bar{z}$, the following holds:
\begin{equation*}
\lim_{i\to+\infty}\dist\parentheses{\bar{v},\partial_{\zeta}g(z_i)}=\dist\parentheses{\bar{v},\partial_{\zeta}g(\bar{z})}=0.
\end{equation*}
\end{proposition}
\begin{proof}
It is sufficient to show that there exists $N\in\Natural$ such that
\begin{equation}\label{eq:inclusion:Nk}
(\forall i\geq N)\;\; \bar{v}\in\partial_{\zeta} g(z_i).
\end{equation}
Consider the following continuous function over $\dom{g}$:
\begin{equation*}
D(z)\coloneqq g(\bar{z})-g(z)-\innerp{\bar{v}}{\bar{z}-z}+\zeta.
\end{equation*}
Since $D(\bar{z})=\zeta>0$ and $z_i\to \bar{z}$, there exists $N\in\Natural$ such that
\begin{equation*}
(\forall i\geq N)\;\; D(z_i)\geq 0.
\end{equation*}
Expanding the expression of $D(z_i)$ yields that
\begin{equation*}
(\forall i\geq N)\;\;g(\bar{z})\geq g(z_i)+\innerp{\bar{v}}{\bar{z}-z_i}-\zeta.
\end{equation*}
Since $\bar{v}\in\partial g(\bar{z})$, the preceding inequality implies that
\begin{equation*}
(\forall i\geq N)\;\; \bar{v}\in\partial_{\zeta} g(z_i).
\end{equation*}
from Proposition \ref{pp:transportation:formula}, which completes the proof.
\end{proof}
Generalizing the proposition above yields a certain continuity property of $\epsilon$-subdifferential. As will be shown in Proposition \ref{pp:constraint:xi:finite}, this continuity property essentially ensures that (\ref{eq:constraint:xik}), i.e., the counterpart of (\ref{eq:constraint:xik:souza}) adopting $\epsilon$-subdifferential, can be satisfied within finite inner iterations.
\begin{theorem}[Upper semicontinuity of $\dist\parentheses{u,\partial_{\zeta}g(\cdot)}$]\label{thm:upper:semicont}
Suppose that $\zeta>0$, $u\in\Real^n$. Then $\dist\parentheses{u,\partial_{\zeta}g(\cdot)}$ is upper semicontinuous over $\dom{g}$, i.e., for any sequence $(z_i)_{i\in\Natural}$ in $\dom{g}$ convergent to $\bar{z}\in\dom{g}$, the following holds:
\begin{equation*}
\limsup_{i\to+\infty}\dist\parentheses{u,\partial_{\zeta}g(z_i)}\leq \dist\parentheses{u,\partial_{\zeta}g(\bar{z})}.
\end{equation*}
\end{theorem}
\begin{proof}
Since $\partial_{\zeta}g(\bar{z})$ is a nonempty closed convex set (Proposition \ref{pp:convexity:esubdiff}), the projection theorem \cite[Thm. 3.16]{bauschke2017} ensures the existence of $\bar{u}\in\partial_{\zeta}g(\bar{z})$ satisfying
\begin{equation*}
\dist\parentheses{u,\partial_{\zeta}g(\bar{z})}=\norm{u-\bar{u}}.
\end{equation*}
Similarly, for every $i\in\Natural$, there exists $u_i\in\partial_{\zeta}g(z_i)$ such that
\begin{equation*}
\dist\parentheses{\bar{u},\partial_{\zeta}g(z_i)}=\norm{\bar{u}-u_i},
\end{equation*}
which implies
\begin{equation}\label{eq:upper:cont:a}
\dist\parentheses{u,\partial_{\zeta}g(z_i)} \leq\norm{u-u_i}\leq \norm{\bar{u}-u_i}+\norm{u-\bar{u}}\leq \dist\parentheses{\bar{u},\partial_{\zeta}g(z_i)}+\norm{u-\bar{u}}. 
\end{equation}
Since $\lim_{i\to +\infty}\dist\parentheses{\bar{u},\partial_{\zeta}g(z_i)}=0$ from Proposition \ref{pp:continuity:dist:esubdiff}, taking the limit superior $i\to +\infty$ on both sides of (\ref{eq:upper:cont:a}) yields
\begin{equation*}
\limsup_{i\to +\infty} \dist\parentheses{u,\partial_{\zeta}g(z_i)}\leq \norm{u-\bar{u}}=\dist\parentheses{u,\partial_{\zeta}g(\bar{z})},
\end{equation*}
which completes the proof.
\end{proof}

\subsubsection{Proof of Finite Inner Loop Termination}
To prove the finite termination of the step 2, we only need to show that for any $T_g$ satisfying Assumption \ref{assumpt:solution:steptwo}, there exists some $N_k\in\braces{-1}\cup\Natural$, such that $z_{N_k}$ satisfies both (\ref{eq:decrease:g}) and (\ref{eq:constraint:xik}). In the sequel, we deal with these two conditions separately and prove that both of them are satisfied for large enough $i$, under the assumption that $x_k$ does not coincide with a solution of (PP).

\begin{proposition}[Achievability of (\ref{eq:decrease:g})]\label{pp:decrease:g:finite}
Suppose that $x_k$ is not a solution of (PP), then (\ref{eq:decrease:g}) remains satisfied after finite iterations of $T_g(x_k,u_k,\lambda_k)$, i.e., there exists $N_k\in \Natural$ such that
\begin{equation*}
(\forall i\geq N_k)\;\;g(x_k)-g(z_{i})-\innerp{u_k}{x_k-z_{i}}\geq \frac{(1-\sigma)}{\lambda}\norm{z_{i}-x_k}^2.
\end{equation*}
\end{proposition}
\begin{proof}
From Assumption \ref{assumpt:solution:steptwo}, $0\not\in \tilde{f}_k(x_k)$ and $\lambda_k=\lambda$, there exists $\bar{z}_{k}(\neq x_k)\in\dom{g}$ such that $u_k-\frac{1}{\lambda}(\bar{z}_{k}-x_k)\in\partial g(\bar{z}_{k})$, thus we have
\begin{equation*}
g(x_k)\geq g(\bar{z}_{k})+\innerp{u_k-\frac{1}{\lambda}(\bar{z}_{k}-x_k)}{x_k-\bar{z}_{k}}, 
\end{equation*}
which can be transformed into
\begin{equation*}
g(x_k)-g(\bar{z}_{k})-\innerp{u_k}{x_k-\bar{z}_{k}}-\frac{(1-\sigma)}{\lambda}\norm{\bar{z}_{k}-x_k}^2\geq \frac{\sigma}{\lambda}\norm{\bar{z}_{k}-x_k}^2>0.
\end{equation*}
Now consider the following continuous function $D(z)$ over $\dom{g}$:
\begin{equation*}
D(z)=g(x_k)-g(z)-\innerp{u_k}{x_k-z}-\frac{(1-\sigma)}{\lambda}\norm{z-x_k}^2.
\end{equation*}
Since $D(\bar{z}_{k})\geq \frac{\sigma}{\lambda}\norm{\bar{z}_{k}-x_k}^2>0$ and $z_i\to \bar{z}_{k}$, there exists $N_{k}\in\Natural$ such that
\[(\forall i\geq N_k)\;\; D(z_{i})\geq 0.\]
Expanding the expression of $D(z_i)$ yields
\begin{equation*}
(\forall i\geq N_k)\;\;g(x_k)-g(z_{i})-\innerp{u_k}{x_k-z_{i}}\geq \frac{(1-\sigma)}{\lambda}\norm{z_{i}-x_k}^2,
\end{equation*}
which completes the proof.
\end{proof}

\begin{proposition}[Achievability of (\ref{eq:constraint:xik})]\label{pp:constraint:xi:finite}
Suppose that $x_k$ is not a solution of (PP), then (\ref{eq:constraint:xik}) remains satisfied after finite iterations of $T_g(x_k,u_k,\lambda_k)$, i.e., there exists $ N_k\in \braces{-1}\cup \Natural$ such that
\begin{equation*}
(\forall i\geq N_k)\;\;\dist(u_k,\partial_{\zeta_k}g(z_{i}))\leq \theta\norm{z_{i}-x_k}.
\end{equation*}
\end{proposition}
\begin{proof}
From Assumption \ref{assumpt:solution:steptwo}, $0\not\in \tilde{f}_k(x_k)$ and $\lambda_k=\lambda$, there exists $\bar{z}_{k}(\neq x_k)\in \dom{g}$ such that $u_k-\frac{1}{\lambda}(\bar{z}_{k}-x_k)\in\partial g(\bar{z}_{k})$.
Consider the following function over $\dom{g}$:
\begin{equation*}
D(z)\coloneqq \dist\parentheses{u_k,\partial_{\zeta_k}g(z)}-\theta\norm{z-x_k}.
\end{equation*}
Since $u_k-\frac{1}{\lambda}(\bar{z}_{k}-x_k)\in\partial g(\bar{z}_{k})\subset\partial_{\zeta_k}g(\bar{z}_{k})$, we have
\[\dist\parentheses{u_k,\partial_{\zeta_k}g(\bar{z}_{k})}\leq \frac{1}{\lambda}\norm{\bar{z}_{k}-x_k}\]
thus
\begin{equation*}
D(\bar{z}_{k})\leq \parentheses{\frac{1}{\lambda}-\theta}\norm{\bar{z}_{k}-x_k}<0.
\end{equation*}
Since $z_i\to\bar{z}_{k}$, the upper semicontinuity of $D(z)$ (Theorem \ref{thm:upper:semicont}) then yields
\begin{equation*}
\limsup_{i\to+\infty}D(z_i)\leq D(\bar{z}_{k})<0.
\end{equation*}
Therefore, there exists $N_k\in\Natural$ such that
\begin{equation*}
(\forall i\geq N_k)\;\; D(z_i)\leq 0.
\end{equation*}
Expanding the expression of $D(z_i)$ yields
\begin{equation*}
(\forall i\geq N_k)\;\; \dist\parentheses{u_k,\partial_{\zeta_k}g(z_i)}\leq\theta\norm{z_i-x_k},
\end{equation*}
which completes the proof.
\end{proof}

Applying the preceding two propositions, we can prove the finite inner loop termination of tPLDCA.

\begin{theorem}[Inner loop termination of tPLDCA]\label{thm:finite:termination}
The step 2 of tPLDCA terminates within finite inner iterations, i.e., there exists $N_k\in \Natural$ such that $z_{N_k}$ satisfies simultaneously
\begin{equation}\label{eq:decrease:g:finite}
g(x_k)-g(z_{N_k})-\innerp{u_k}{x_k-z_{N_k}}\geq \frac{(1-\sigma)}{\lambda}\norm{z_{N_k}-x_k}^2
\end{equation}
and
\begin{equation}\label{eq:constraint:xi:finite}
\dist(u_k,\partial_{\zeta_k}g(z_{N_k}))\leq \theta\norm{z_{N_k}-x_k}.
\end{equation}
\end{theorem}
\begin{proof}
From Assumption \ref{assumpt:solution:steptwo} and $\lambda_k=\lambda$, there exists $\bar{z}_{k}\in\dom{g}$ such that 
\[u_k-\frac{1}{\lambda}(\bar{z}_{k}-x_k)\in\partial g(\bar{z}_{k}).\]
\begin{enumerate}
\item[(I)] If $\bar{z}_{k}=x_k$, then since $z_{-1}=x_k$, we have
\begin{equation*}
g(x_k)-g(z_{-1})-\innerp{u_k}{x_k-z_{-1}}=0\geq \frac{(1-\sigma)}{\lambda}\norm{z_{-1}-x_k}^2=0.
\end{equation*}
In addition, since $z_{-1}=x_k=\bar{z}_{k}$, we have $u_k\in \partial g(z_{-1})\subset\partial_{\zeta_k}g(z_{-1})$, thus
\begin{equation*}
\dist\parentheses{u_k,\partial_{\zeta_k}g(z_{-1})}=0\leq \theta\norm{z_{-1}-x_k}^2=0.
\end{equation*}
Therefore, $N_k=0$ satisfies (\ref{eq:decrease:g:finite}) and (\ref{eq:constraint:xi:finite}).

\item[(II)] If $\bar{z}_{k}\neq x_k$. Then from Proposition \ref{pp:decrease:g:finite}, there exists $N_{1,k}\in\Natural$ such that
\begin{equation*}
(\forall i\geq N_{1,k})\;\; g(x_k)-g(z_i)-\innerp{u_k}{x_k-z_i}\geq \frac{(1-\sigma)}{\lambda}\norm{z_i-x_k}.
\end{equation*}
From Proposition \ref{pp:constraint:xi:finite}, there exists $N_{2,k}\in\Natural$ such that
\begin{equation*}
(\forall i\geq N_{2,k})\;\; \dist(u_k,\partial_{\zeta_k}g(z_i))\leq \theta \norm{z_i-x_k}.
\end{equation*}
Therefore, $N_k\coloneqq \max\{N_{1,k},N_{2,k}\}$ satisfies (\ref{eq:decrease:g:finite}) and (\ref{eq:constraint:xi:finite}).
\end{enumerate}
Combining (I) and (II) completes the proof.
\end{proof}

\subsection{Convergence Properties}
The proposed tPLDCA introduces relaxations in both step 1 and 2 to permit various calculation errors. In this section, we show that despite such relaxations, tPLDCA enjoys the same convergence properties as that of exact DCA \cite{pham1997} and PLDCA \cite{sun2003}.

First, we prove the convergence of functional values of the iterates of tPLDCA.
\begin{theorem}[Convergence of functional values of iterates]\label{thm:convergence:functional}
The sequence $(x_k)_{k\in\Natural}$ generated by tPLDCA satisfies:
\begin{enumerate}
\item[(a)] $\parentheses{f(x_k)+(\frac{\gamma}{2}+\rho)\norm{x_k-x_{k-1}}^2}_{k\in\Natural}$ is decreasing.

\item[(b)] $\lim_{k\to +\infty}\norm{x_{k+1}-x_k}=0$.

\item[(c)] $(f(x_k))_{k\in\Natural}$ is convergent.
\end{enumerate}
\end{theorem}
\begin{proof}
(a) Since $\dist(u_k,\partial_{\epsilon_k}h(x_k))\leq d_k$ implies the existence of $\bar{u}_k\in\partial_{\epsilon_k}h(x_k)$ satisfying $\norm{\bar{u}_k-u_k}\leq d_k$, we have
\begin{equation*}
h(x_{k+1})\geq h(x_k)+\innerp{\bar{u}_k}{x_{k+1}-x_k}-\epsilon_k,
\end{equation*}
which can be rewritten as
\begin{equation*}
h(x_{k+1})\geq h(x_k)+\innerp{u_k}{x_{k+1}-x_k}+\innerp{\bar{u}_k-u_k}{x_{k+1}-x_k}-\epsilon_k.
\end{equation*}
Combining the inequality above with
\begin{equation*}
\innerp{\bar{u}_k-u_k}{x_{k+1}-x_k}\geq -\norm{\bar{u}_k-u_k}\norm{x_{k+1}-x_k}\geq -d_k\norm{x_{k+1}-x_k}
\end{equation*}
yields
\begin{equation}\label{eq:decrease:h}
h(x_{k+1})\geq h(x_k)+\innerp{u_k}{x_{k+1}-x_k}-d_k\norm{x_{k+1}-x_k}-\epsilon_k.
\end{equation}

On the other hand, summing (\ref{eq:decrease:h}) with (\ref{eq:decrease:g}) yields
\begin{equation}\label{eq:decrease:f}
f(x_k)\geq f(x_{k+1})-d_k\norm{x_{k+1}-x_k}-\epsilon_k+\frac{(1-\sigma)}{\lambda}\norm{x_{k+1}-x_k}^2.
\end{equation}
Considering that $d_k=\gamma\norm{x_k-x_{k-1}}$, we have
\begin{equation*}
d_k\norm{x_{k+1}-x_k}=\gamma\norm{x_k-x_{k-1}}\norm{x_{k+1}-x_k}\leq \frac{\gamma}{2}\norm{x_k-x_{k-1}}^2+\frac{\gamma}{2}\norm{x_{k+1}-x_k}^2.
\end{equation*}
Substituting the inequality above and $\epsilon_k=\rho\norm{x_k-x_{k-1}}^2$ into (\ref{eq:decrease:f}) yields
\begin{equation*}
f(x_k)\geq f(x_{k+1})+\parentheses{-\frac{\gamma}{2}+\frac{(1-\sigma)}{\lambda}}\norm{x_{k+1}-x_k}^2-\parentheses{\frac{\gamma}{2}+\rho}\norm{x_k-x_{k-1}}^2,
\end{equation*}
which can be transformed into
\begin{equation}\label{eq:decrease:fplusnorm}
f(x_k)+\parentheses{\frac{\gamma}{2}+\rho}\norm{x_k-x_{k-1}}^2\geq f(x_{k+1})+\parentheses{\frac{\gamma}{2}+\rho}\norm{x_{k+1}-x_k}^2+\eta\norm{x_{k+1}-x_k}^2,
\end{equation} 
where $\eta=(1-\sigma)/\lambda-\gamma-\rho$. Since $\gamma\in[0,(1-\sigma)/\lambda-\rho)$, we conclude
\begin{equation*}
f(x_k)+\parentheses{\frac{\gamma}{2}+\rho}\norm{x_k-x_{k-1}}^2\geq f(x_{k+1})+\parentheses{\frac{\gamma}{2}+\rho}\norm{x_{k+1}-x_k}^2,
\end{equation*}
which completes the proof of (a).
\bigskip

\noindent
(b) By summing up (\ref{eq:decrease:fplusnorm}) we obtain
\begin{equation*}
f(x_0)-f(x_{k+1})\geq -\parentheses{\frac{\gamma}{2}+\rho}\norm{x_{0}-x_{-1}}^2+\parentheses{\frac{\gamma}{2}+\rho}\norm{x_{k+1}-x_k}^2+\eta\sum_{i=0}^k \norm{x_{i+1}-x_i}^2.
\end{equation*}
Therefore, it can be derived that
\begin{equation*}
\eta\sum_{i=0}^{+\infty} \norm{x_{i+1}-x_i}^2\leq f(x_0)-\inf_{x\in\Real^n}\braces{f(x)}+\parentheses{\frac{\gamma}{2}+\rho}\norm{x_0-x_{-1}}^2<+\infty,
\end{equation*}
which implies $\lim_{k\to +\infty}\norm{x_{k+1}-x_k}=0$.
\bigskip

\noindent
(c) Combining (a) and 
\begin{equation*}
f(x_k)+\parentheses{\frac{\gamma}{2}+\rho}\norm{x_k-x_{k-1}}^2\geq \inf_{x\in\Real^n}\braces{f(x)}+0>-\infty,
\end{equation*}
we yield that $\parentheses{f(x_k)+\parentheses{\frac{\gamma}{2}+\rho}\norm{x_k-x_{k-1}}^2}_{k\in\Natural}$ is decreasing and lower bounded, thus is convergent. Suppose that
\begin{equation*}
\lim_{k\to +\infty} \braces{f(x_k)+\parentheses{\frac{\gamma}{2}+\rho}\norm{x_k-x_{k-1}}^2}=\alpha,
\end{equation*}
then since $\lim_{k\to +\infty}\norm{x_k-x_{k-1}}=0$ from (b), we obtain that 
\[\lim_{k\to +\infty} f(x_k)=\alpha,\]
thus $\parentheses{f(x_k)}_{k\in\Natural}$ is convergent.
\end{proof}

Next, we prove that every cluster point (if exists) of the iterates generated by tPLDCA is a critical point of $f$, which is a celebrated convergence guarantee that emerged in the classical studies of DC-programming \cite{pham1997,le2005}.
\begin{theorem}[Convergence to a critical point]\label{thm:convergence:point}
Suppose that $(x_k)_{k\in\Natural}$ is a sequence generated by tPLDCA and $\bar{x}$ is a cluster point of it. Then the following hold:
\begin{enumerate}
\item[(a)] $\lim_{k\to +\infty} f(x_k)=f(\bar{x})$.

\item[(b)] $\bar{x}$ is a critical point of $f$, i.e., $\partial g(\bar{x})\cap \partial h(\bar{x})\neq \emptyset$.
\end{enumerate}
\end{theorem}
\begin{proof}
(a) Since $\bar{x}$ is a cluster point of $(x_k)_{k\in\Natural}$, we can choose a subsequence $(x_{k_i})_{i\in\Natural}$ of $(x_k)_{k\in\Natural}$ satisfying $\lim_{i\to +\infty}x_{k_i}=\bar{x}$. From $x_0\in\dom{f}$, (\ref{eq:decrease:g}) and Lemma \ref{lemma:domf}, $x_{k}\in\dom{g}=\dom{f}$ for every $k\in\Natural$. Since $\dom{f}=C$ is closed (Lemma \ref{lemma:domf}), we have $\bar{x}\in\dom{f}$. The continuity of $f$ over $\dom{f}$ (Lemma \ref{lemma:contghf}) yields
\begin{equation*}
\lim_{i\to +\infty} f(x_{k_i})=f(\bar{x}).
\end{equation*}
The convergence of $\parentheses{f(x_k)}_{k\in\Natural}$ (Theorem \ref{thm:convergence:functional}(c)) implies
\begin{equation*}
\lim_{k\to +\infty} f(x_{k})=\lim_{i\to +\infty} f(x_{k_i})=f(\bar{x}),
\end{equation*}
which completes the proof of (a).
\bigskip

\noindent
(b) Let $(x_{k_i})_{i\in\Natural}$ be a subsequence of $(x_k)_{k\in\Natural}$ convergent to $\bar{x}$. Since $\lim_{k\to +\infty}\norm{x_{k+1}-x_k}=0$ (Theorem \ref{thm:convergence:functional}(b)), $(\epsilon_k)_{k\in\Natural}$ and $(d_k)_{k\in\Natural}$ are both bounded. Let $\bar{\epsilon}$ be an upper bound of $(\epsilon_k)_{k\in\Natural}$ and let $\bar{d}$ be an upper bound of $(d_k)_{k\in\Natural}$, then it can be verified that
\begin{equation*}
\parentheses{\forall i\in\Natural}\;\;\dist(u_{k_i},\partial_{\bar{\epsilon}}h(x_{k_i}))\leq \bar{d}.
\end{equation*}
Since $(x_{k_i})_{i\in\Natural}$ is bounded, 
\begin{equation*}
\partial_{\bar{\epsilon}}h((x_{k_i})_{i\in\Natural})\coloneqq \cup_{i\in\Natural} \partial_{\bar{\epsilon}}h(x_{k_i})
\end{equation*}
is bounded from Lemma \ref{lemma:bounded:partialhS}, thus $(u_{k_i})_{i\in\Natural}$ is also bounded. Therefore, by further passing to a subsequence, if necessary, we can assume the existence of some $\bar{u}\in\Real^n$ such that
\begin{equation*}
\lim_{i\to +\infty} x_{k_i} = \bar{x},\;\lim_{k\to +\infty} u_{k_i} = \bar{u}.
\end{equation*}
Moreover, for every $k\in\Natural$, by noting $\dist(u_k,\partial_{\zeta_k}g(x_{k+1}))\leq\theta\norm{x_{k+1}-x_k}$ implies the existence of $\xi_{k+1}\in\partial_{\zeta_k}g(x_{k+1})$ satisfying 
\begin{equation*}
\norm{\xi_{k+1}-u_{k}}\leq \theta\norm{x_{k+1}-x_k},
\end{equation*}
the property $\lim_{k\to +\infty} \norm{x_{k+1}-x_k}=0$ (by Theorem \ref{thm:convergence:functional}) ensures
\begin{align*}
\lim_{i\to +\infty} x_{k_i+1} &=\lim_{i\to +\infty}\braces{x_{k_i}+(x_{k_i+1}-x_{k_i})}=\lim_{i\to +\infty}x_{k_i}=\bar{x},\\
\lim_{i\to +\infty} \xi_{k_i+1} &=\lim_{i\to+\infty} \braces{u_{k_i}+(\xi_{k_i+1}-u_{k_i})}=\lim_{i\to +\infty} u_{k_i}=\bar{u}.
\end{align*}
In addition, since $\dist(u_{k},\partial_{\epsilon_{k}}h(x_{k}))\leq d_{k}$, we can obtain
\begin{equation*}
(\forall y\in\Real^n)\;\; h(y)\geq h(x_{k_i})+\innerp{u_{k_i}}{y-x_{k_i}}-d_{k_i}\norm{y-x_{k_i}}-\epsilon_{k_i}
\end{equation*}
through a similar derivation with (\ref{eq:decrease:h}). Since $\xi_{k+1}\in\partial_{\zeta_k} g(x_{k+1})$, 
\begin{equation*}
(\forall y\in\Real^n)\;\; g(y)\geq g(x_{k_i+1})+\innerp{\xi_{k_i+1}}{y-x_{k_i+1}}-\zeta_{k_i}.
\end{equation*}
Since $g,h$ are continuous over $\dom{f}$ (Lemma \ref{lemma:contghf}), taking $i\to +\infty$ in the two inequalities above yields
\begin{align*}
(\forall y\in\Real^n)\;\;h(y)\geq h(\bar{x})+\innerp{\bar{u}}{y-\bar{x}},\\
(\forall y\in\Real^n)\;\;g(y)\geq g(\bar{x})+\innerp{\bar{u}}{y-\bar{x}},
\end{align*}
which implies $\bar{u}\in\partial h(\bar{x})\cap \partial g(\bar{x})\neq\emptyset$. Therefore, $\bar{x}$ is a critical point of $g-h$.
\end{proof}
 It should be noted that in some recent studies, stronger convergence guarantees can be established if $f$ possesses some favourable properties (e.g., decomposition of $h$ as maximum of convex smooth functions \cite{pang2017}, KL property of some auxiliary function involving $f$ \cite{wen2018}). We believe that similar extensive analysis is possible for tPLDCA and leave it for future research.

\section{A Computational Friendly Surrogate of $\epsilon$-Subdifferential}
\label{sec:subdifferential}
A noticeable feature of tPLDCA is the use of $\epsilon$-subdifferential. However, from Definition \ref{def:esubdiff}, computing the whole set of $\epsilon$-subdifferential can be difficult in practice. In this section, we first analyse possible computational difficulties in tPLDCA regarding calculation of $\epsilon$-subdifferential. Then, in case such difficulties occur, we propose the notion of $\epsilon$-strict subdifferential under an additional assumption on $g$, and we show that $\epsilon$-strict subdifferential serves as a computational friendly surrogate of $\epsilon$-subdifferential. As will be shown in Sec. \ref{sec:implementation}, employing such surrogate leads to a pragmatic implementation of tPLDCA without losing all of its favourable properties.

\subsection{Computational Difficulties in tPLDCA regarding $\epsilon$-Subdifferential}
\label{subsec:implementability:CiPLDCA}

In tPLDCA, the step 1 finds an approximation of an element in $\partial_{\epsilon_k}h(x_k)$. Since this does not require computing the whole set of $\partial_{\epsilon_k}h(x_k)$, as long as we have an oracle that returns an $\epsilon$-subgradient of $h$ corrupted with mild noise, the step 1 is achievable. This applies to many scenarios encountered in practice. See the following remark.

\begin{remark}\label{remark:when_is_step1_implementable}
\begin{enumerate}
\item[(a)] If $\partial h(x_k)$ has a closed-form expression (e.g., $h$ is a simple differentiable function or a simple nonsmooth function such as the $l_1$-norm), then every $u_k\in\partial h(x_k)$ satisfies (\ref{eq:constraint:uk}).

\item[(b)] If $h$ can be approximated by a subdifferentiable function $\bar{h}\in\Gamma_0(\Real^n)$ with approximation error no larger than $\epsilon_k/2$, i.e.,
\begin{equation*}
(\forall x\in\Real^n)\;\; \lvert h(x)-\bar{h}(x) \rvert \leq \epsilon_k/2,
\end{equation*}
and $\partial \bar{h}(x_k)$ has a closed-form expression, then for every $u_k\in \partial \bar{h}(x_k)$, we have
\begin{equation*}
(\forall x\in\Real^n)\;\; h(x)\geq \bar{h}(x)-\epsilon_k/2\geq \bar{h}(x_k)+\innerp{u_k}{x-x_k}-\epsilon_k/2\geq h(x_k)-\epsilon_k+\innerp{u_k}{x-x_k},
\end{equation*}
which implies that $u_k\in\partial_{\epsilon_k}h(x_k)$, hence $u_k$ satisfies (\ref{eq:constraint:uk}).

\item[(c)] If $h$ can be constructed by conducting simple operations (e.g., composition with affine operator, taking sum/maximum) on a finite family of $\epsilon$-subdifferentiable functions $(\bar{h}_i(x))_{i=1}^K$, then one can obtain $u_k\in\partial_{\epsilon_k}h(x_k)$ using $\bar{u}_{k,i}\in\partial_{\epsilon_k}\bar{h}_i(x_k)$ and the calculus rules introduced in \cite[Ch. XI, Sec. 3]{hiriart1993}.
\end{enumerate}
\end{remark}

Therefore, the step 1 of tPLDCA does not pose a computational difficulty in general.

On the other hand, the step 2 of tPLDCA requires checking (\ref{eq:decrease:g}) and (\ref{eq:constraint:xik}) for every inner iterate $z_i$. (\ref{eq:decrease:g}) only involves computing the functional value of $g$, thus is implementable as long as $g(x_k)$ and $g(z_i)$ can be computed in closed-form. However, for given $u_k,z_i\in\Real^n$, computing the distance between $u_k$ and $\partial_{\zeta_k}g(z_i)$ must require the information of the whole set of $\partial_{\zeta_k}g(z_i)$. For very simple $g$, it is possible that $\dist(u_k,\partial_{\zeta_k}g(z_i))$ can be computed exactly. For example, if $g(x)\coloneqq \innerp{b}{x}+c$ with $b\in\Real^n$ and $c\in\Real$, then $\partial_{\zeta_k}g(z_i)=\braces{b}$ for every $\zeta_k>0$ and $z_i\in\Real^n$, hence $\dist(u_k,\partial_{\zeta_k}g(z_i))=\norm{u_k-b}$ is computable. However, in most cases, the computation of $\dist(u_k,\partial_{\zeta_k}g(z_i))$ is difficult, even for a simple quadratic function $g$. See the following example.

\begin{example}
Suppose that $g(x)\coloneqq \frac{1}{2}\innerp{Qx}{x}+\innerp{b}{x}$ with $Q\in\Real^{n\times n}$ being positive semidefinite and $b\in\Real^n$, then by \cite[Ch. XI, Example 1.2.2]{hiriart1993},
\begin{equation*}
\partial_{\zeta_k}g(z_i)=\braces{u\in\Real^n\mid \frac{1}{2}\innerp{u-b}{Q^{\dagger}(u-b)}-\innerp{z_i}{u}\leq \zeta_k-g(z_i)}\cap (b+\mathrm{Im}\,Q^{\dagger}),
\end{equation*}
where $Q^{\dagger}$ is the Moore-Penrose pseudo-inverse of $Q$, $\mathrm{Im}\,Q^{\dagger}$ is the image of $Q^{\dagger}$. Thus $\partial_{\zeta_k}g(z_i)$ is the intersection between an ellipsoid and an affine space, which implies that $\dist(u_k,\partial_{\zeta_k}g(z_i))$ cannot be computed exactly within finite iterations in general.
\end{example}

Therefore, the main computational difficulty in tPLDCA consists in its step 2, more precisely, in the computation of $\dist\parentheses{u_k,\partial_{\zeta_k}g(z_i)}$. 

\subsection{$\epsilon$-Strict Subdifferential: Definition and Inclusion Relation}
In order to resolve the difficulty of computing $\dist\parentheses{u_k,\partial_{\zeta_k}g(z_i)}$, in this section we put forward a novel notion of inexact subdifferential named $\epsilon$-strict subdifferential under the following additional assumption on $g$. 
\begin{assumption}\label{assumpt:decomp:g}
The function $g$ can be expressed as the pointwise maximum of finitely many convex smooth functions, i.e.,
\begin{equation}
g(x) = \max \{g_1(x),g_2(x),\dots, g_p(x)\},
\end{equation}
where $p\in\Natural$, every function in $\{g_1,\dots,g_p\}$ is smooth and convex.
\end{assumption}
\begin{remark}
Examples satisfying Assumption \ref{assumpt:decomp:g} include polyhedral convex functions, continuously differentiable convex functions, etc \cite{le2015,pang2017}.
\end{remark}

Under Assumption \ref{assumpt:decomp:g}, we formally define the $\epsilon$-strict subdifferential.

\begin{definition}[$\epsilon$-strict subdifferential]\label{def:estrict:subdiff}
For $\forall x\in\Real^n$ and $\epsilon\geq 0$, we define the $\epsilon$-strict subdifferential of $g$ at $x$ as
\begin{equation}
\hat{\partial}_\epsilon g(x) \coloneqq \convhull\braces{\nabla g_i(x)\mid i\in I_{g,\epsilon}(x)},
\end{equation}
where $I_{g,\epsilon}(x)\coloneqq \braces{i\in\{1,\dots,p\}\mid \; g_i(x)\geq g(x)-\epsilon}$, and $\convhull\braces{S}$ is the convex hull of a set $S$.
\end{definition}
\begin{remark}
It can be verified that $\hat{\partial}_0 g(x)$ coincides with $\partial g(x)$ \cite[Thm. 18.5 (Dubovitskii-Milyutin)]{bauschke2017}.
\end{remark}

Remarkably, the following inclusion relation holds among subdifferential, $\epsilon$-strict subdifferential and $\epsilon$-subdifferential.

\begin{proposition}[Inclusion relation]\label{pp:subset:hatpartial}
For every $ x\in\Real^n$ and $\epsilon>0$, the following inclusion holds: 
\[\partial g(x)\subset \hat{\partial}_{\epsilon}g(x)\subset \partial_{\epsilon}g(x).\]
\end{proposition}
\begin{proof}
We first show $\partial g(x)\subset \hat{\partial}_{\epsilon}g(x)$. By \cite[Thm. 18.5 (Dubovitskii-Milyutin)]{bauschke2017},
\begin{equation*}
\partial g(x)=\convhull\braces{\nabla g_i(x)\mid i\in I_{g,0}(x)},
\end{equation*}
hence $\partial g(x)\subset \hat{\partial}_{\epsilon}g(x)$ from $I_{g,0}(x)\subset I_{g,\epsilon}(x)$.

Next we show $\hat{\partial}_{\epsilon}g(x)\subset\partial_{\epsilon}g(x)$. From Definition \ref{def:estrict:subdiff}, for any $u\in\hat{\partial}_{\epsilon}g(x)$, there exists a set of nonnegative numbers $\braces{c_i\geq 0\mid i\in I_{g,\epsilon}(x)}$ such that
\begin{equation}\label{eq:estrict:inclusion:a}
u=\sum\nolimits_{i\in I_{g,\epsilon}(x)} c_i \nabla g_i(x),
\end{equation}
where
\begin{equation}\label{eq:estrict:inclusion:b}
\sum\nolimits_{i\in I_{g,\epsilon}(x)} c_i=1.
\end{equation}
From Assumption \ref{assumpt:decomp:g} and convexity of $g_i$, we have
\begin{equation*}
\parentheses{\forall i\in I_{g,\epsilon}(x),\forall y\in\Real^n}\;\;g(y)\geq g_i(y)\geq g_i(x)+\innerp{\nabla g_i(x)}{y-x}
\end{equation*}
Substituting 
\[\parentheses{\forall i\in I_{g,\epsilon}(x)}\;\;g_i(x)\geq g(x)-\epsilon\]
into the inequality above yields
\begin{equation*}
\parentheses{\forall i\in I_{g,\epsilon}(x),\forall y\in\Real^n}\;\;g(y)\geq g(x)-\epsilon+\innerp{\nabla g_i(x)}{y-x}.
\end{equation*}
Therefore, the following holds
\begin{align*}
\parentheses{\forall y\in\Real^n}\;\;\sum_{i\in I_{g,\epsilon}(x)} c_i g(y) \geq \sum_{i\in I_{g,\epsilon}(x)} c_i\brackets{ g(x)-\epsilon+\innerp{\nabla g_i(x)}{y-x}}.
\end{align*}
Substituting (\ref{eq:estrict:inclusion:a}) and (\ref{eq:estrict:inclusion:b}) into the inequality above yields
\begin{equation*}
\parentheses{\forall y\in\Real^n}\;\; g(y)\geq g(x)-\epsilon+\innerp{u}{y-x},
\end{equation*}
which implies $u\in\partial_{\epsilon}g(x)$. Hence $\hat{\partial}_{\epsilon}g(x)\subset\partial_{\epsilon}g(x)$.

Combining the discussion above completes the proof.
\end{proof}
\subsection{$\epsilon$-Strict Subdifferential: Favourable Properties}
In the following, we present two favourable properties of $\epsilon$-strict subdifferential. The first one shows that, $\epsilon$-strict subdifferential enjoys computational advantages over $\epsilon$-subdifferential.

\begin{remark}[Computability of projection onto $\epsilon$-strict subdifferential]\label{remark:computability:estrict}~

\noindent
For given $x\in\Real^n$ and $\epsilon>0$, the set $\hat{\partial}_{\epsilon}g(x)$ is defined as a convex polyhedron with known vertices $\nabla g_i(x)\; (i\in I_{g,\epsilon}(x))$. Therefore, the projection onto $\hat{\partial}_{\epsilon}g(x)$ can be computed via finitely terminating algorithms such as Wolfe's algorithm \cite{wolfe1976,gabidullina2018}. 
\end{remark}

Secondly, compared to subdifferential, $\epsilon$-strict subdifferential enjoys the same continuity property (Theorem \ref{thm:upper:semicont}) as $\epsilon$-subdifferential does. To see this, we first establish a counterpart of Proposition \ref{pp:continuity:dist:esubdiff} for $\epsilon$-strict subdifferential.

\begin{proposition}\label{pp:continuity:dist:estict}
Suppose that $\zeta>0$, $\bar{z}\in\Real^n$ and $\bar{v}\in\partial g(\bar{z})$. Then for any sequence $(z_i)_{i\in\Natural}$ convergent to $\bar{z}$, the following holds:
\begin{equation*}
\lim_{i\to+\infty}\dist\parentheses{\bar{v},\hat{\partial}_{\zeta}g(z_i)}=\dist\parentheses{\bar{v},\hat{\partial}_{\zeta}g(\bar{z})}=0.
\end{equation*}
\end{proposition}
\begin{proof}
(I) First, we show the existence of some $N\in\Natural$ satisfying
\begin{equation}\label{eq:continuity:estrict:a}
(\forall i\geq N)\;\;I_{g,0}(\bar{z})\subset I_{g,\zeta}(z_i).
\end{equation}
For $ j\in I_{g,0}(\bar{z})$ and $l\in\{1,\dots, p\}$, consider the following function
\begin{equation*}
s_{jl}(z)=g_j(z)-g_l(z).
\end{equation*}
Since $j\in I_{g,0}(\bar{z})$, $s_{jl}(\bar{z})\geq 0$. Since $s_{jl}(z)$ is continuous and $z_i\to \bar{z}$, there exists $N_{jl}\in\Natural$ such that
\begin{equation*}
\parentheses{\forall i\geq N_{jl}}\;\; s_{jl}(z_i)\geq -\zeta.
\end{equation*}
Expanding the expression of $s_{jl}(z_i)$ yields
\begin{equation*}
\parentheses{\forall i\geq N_{jl}}\;\; g_j(z_i)\geq g_l(z_i)-\zeta.
\end{equation*}
Let $N\coloneqq \max\braces{N_{jl}\mid j\in I_{g,0}(\bar{z}),l\in\braces{1,\dots,p}}$, then the following holds
\begin{equation*}
\parentheses{\forall i\geq N, \forall j\in I_{g,0}(\bar{z}),\forall l\in\braces{1,\dots, p}}\;\; g_j(z_i)\geq g_l(z_i)-\zeta.
\end{equation*}
The arbitrariness of $l\in\braces{1,\dots,p}$ yields that
\begin{equation*}
\parentheses{\forall i\geq N, \forall j\in I_{g,0}(\bar{z})}\;\; j\in I_{g,\zeta}(z_i),
\end{equation*}
which further implies
\begin{equation*}
\parentheses{\forall i\geq N}\;\;  I_{g,0}(\bar{z})\subset I_{g,\zeta}(z_i).
\end{equation*}

\noindent
(II) From \cite[Thm. 18.5 (Dubovitskii-Milyutin)]{bauschke2017}, for any $\bar{v}\in \partial g(\bar{z})$, there exists a set of nonnegative numbers $\braces{c_j\geq 0\mid j\in I_{g,0}(\bar{z})}$ such that
\begin{equation}\label{eq:continuity:estrict:b}
\bar{v}=\sum\nolimits_{j\in I_{g,0}(\bar{z})} c_j \nabla g_j(\bar{z}),
\end{equation}
where
\begin{equation}\label{eq:continuity:estrict:c}
\sum\nolimits_{j\in I_{g,0}(\bar{z})}c_j=1.
\end{equation}
Consider the following continuous vector-valued function:
\begin{equation*}
d(z)\coloneqq \sum\nolimits_{j\in I_{g,0}(\bar{z})} c_j\nabla g_j(z),
\end{equation*}
then (\ref{eq:continuity:estrict:a}) and (\ref{eq:continuity:estrict:c}) implies
\begin{equation*}
\parentheses{\forall i\geq N}\;\; d(z_i)=\sum\nolimits_{j\in \parentheses{I_{g,0}(\bar{z})\subset}I_{g,\zeta}(z_i)} c_j\nabla g_j(z_i)\in \hat{\partial}_{\zeta}g(z_i).
\end{equation*}
Therefore, the following holds
\begin{equation*}
\parentheses{\forall i\geq N}\;\; \dist\parentheses{\bar{v},\hat{\partial}_{\zeta}g(z_i)}\leq\norm{\bar{v}-d(z_i)}
\end{equation*}
Since $\bar{v}=d(\bar{z})$ from (\ref{eq:continuity:estrict:b}), taking the limit $i\to +\infty$ on both sides of the inequality above yields
\begin{equation*}
0\leq \lim_{i\to+\infty}\dist\parentheses{\bar{v},\hat{\partial}_{\zeta}g(z_i)}\leq \lim_{i\to +\infty} \norm{\bar{v}-d(z_i)}=\norm{\bar{v}-d(\bar{z})}=0,
\end{equation*}
which implies
\begin{equation*}
\lim_{i\to+\infty}\dist\parentheses{\bar{v},\hat{\partial}_{\zeta}g(z_i)}=\dist\parentheses{\bar{v},\hat{\partial}_{\zeta}g(\bar{z})}=0.
\end{equation*}
\end{proof}

Applying Proposition \ref{pp:continuity:dist:estict}, we can establish the following continuity property of $\epsilon$-strict subdifferential.

\begin{theorem}[Upper semicontinuity of $\dist\parentheses{u,\hat{\partial}_{\zeta}g(\cdot)}$]\label{thm:upper:semicont:estrict}
Suppose that $\zeta>0$, $u\in\Real^n$. Then $\dist\parentheses{u,\hat{\partial}_{\zeta}g(\cdot)}$ is upper semicontinuous at every $\bar{z}\in\Real^n$, i.e., for any sequence $(z_i)_{i\in\Natural}$ convergent to $\bar{z}\in\Real^n$, the following holds:
\begin{equation*}
\limsup_{i\to+\infty}\dist\parentheses{u,\hat{\partial}_{\zeta}g(z_i)}\leq \dist\parentheses{u,\hat{\partial}_{\zeta}g(\bar{z})}.
\end{equation*}
\end{theorem}
\begin{proof}
Since $\hat{\partial}_{\zeta}g(\bar{z})$ is a nonempty closed convex set, the projection theorem \cite[Thm. 3.16]{bauschke2017} ensures the existence of $\bar{u}\in\hat{\partial}_{\zeta}g(\bar{z})$ satisfying
\begin{equation*}
\dist\parentheses{u,\hat{\partial}_{\zeta}g(\bar{z})}=\norm{u-\bar{u}}.
\end{equation*}
Similarly, for every $i\in\Natural$, there exists $u_i\in\hat{\partial}_{\zeta}g(z_i)$ such that
\begin{equation*}
\dist\parentheses{\bar{u},\hat{\partial}_{\zeta}g(z_i)}=\norm{\bar{u}-u_i},
\end{equation*}
which implies
\begin{equation}\label{eq:upper:cont:estrict:a}
\dist\parentheses{u,\hat{\partial}_{\zeta}g(z_i)} \leq\norm{u-u_i}\leq \norm{\bar{u}-u_i}+\norm{u-\bar{u}}\leq \dist\parentheses{\bar{u},\hat{\partial}_{\zeta}g(z_i)}+\norm{u-\bar{u}}. 
\end{equation}
Since $\lim_{i\to +\infty}\dist\parentheses{\bar{u},\hat{\partial}_{\zeta}g(z_i)}=0$ from Proposition \ref{pp:continuity:dist:estict}, taking the limit superior $i\to +\infty$ on both sides of (\ref{eq:upper:cont:estrict:a}) yields
\begin{equation*}
\limsup_{i\to +\infty} \dist\parentheses{u,\hat{\partial}_{\zeta}g(z_i)}\leq \norm{u-\bar{u}}=\dist\parentheses{u,\hat{\partial}_{\zeta}g(\bar{z})},
\end{equation*}
which completes the proof.
\end{proof}

A comparison among subdifferential, $\epsilon$-strict subdifferential and $\epsilon$-subdifferential is summarized in Table \ref{tab:comparision:subdiff}. From the table, we can see that $\epsilon$-strict subdifferential achieves the best trade-off between computability and continuity property, which in turn ensures the implementability and finite inner loop termination of the resulting implementation of tPLDCA that follows.

\begin{table}
\tbl{Comparison among three notions of subdifferential}
{\begin{tabular}{cccc}
\toprule
 & subdifferential  & $\epsilon$-strict subdifferential & $\epsilon$-subdifferential\\
\midrule
computability &  $\checkmark$  & $\checkmark$  & $\times$ \\
continuity property    &  $\times$ & $\checkmark$  & $\checkmark$ \\
\bottomrule
\end{tabular}}
\label{tab:comparision:subdiff}
\end{table}
\section{A Feasible Implementation of the Proposed tPLDCA}\label{sec:implementation}
In this section, we adopt $\epsilon$-strict subdifferential as a computational friendly surrogate of $\epsilon$-subdifferential, whereby we propose Algorithm \ref{alg:step:two}, which is a feasible implementation of tPLDCA for DC programs satisfying Assumption \ref{assumpt:decomp:g}.

\subsection{Implementability }
In Algorithm \ref{alg:step:two}, in place of ${\partial}_{\zeta_k}g(z_{i})$ in (\ref{eq:constraint:xik}), $\hat{\partial}_{\zeta_k}g(z_{i})$ is employed, to yield (\ref{eq:constraint:xik:IiPLDCA}). Since projection onto $\hat{\partial}_{\zeta_k}g(z_i)$ can be computed efficiently (Remark \ref{remark:computability:estrict}), $\dist\parentheses{u_k,\hat{\partial}_{\zeta_k}g(z_{i})}$ is computable, which resolves the main implementation difficulty faced in tPLDCA (Sec. \ref{subsec:implementability:CiPLDCA}). Therefore, Algorithm \ref{alg:step:two} is an implementable algorithm.

\begin{algorithm}
\textbf{Parameters:} Set $\sigma\in(0,1),\lambda>0,\theta>\frac{1}{\lambda},\rho\in\Big[0,\frac{1-\sigma}{\lambda}\Big),\gamma\in\Big[0,\frac{(1-\sigma)}{\lambda}-\rho\Big)$, $\lambda_k\coloneqq \lambda$ for every $k\in\Natural$ and set a sequence of positive numbers $(\zeta_k)_{k\in\Natural}$ such that $\lim_{k\to +\infty} \zeta_k=0$.

\vspace{1em}
Set $k=0$, $x_{-1}, x_0\in\dom{f}$. Repeat the following steps until convergence.
\vspace{1em}

\textbf{Step 1:} find $u_k\in\Real^n$ such that
\begin{equation}\tag{\ref{eq:constraint:uk}}
\dist\parentheses{u_k,\partial_{\epsilon_k}h(x_k)}\leq d_k,
\end{equation}
$\quad\quad\quad\;\;\;$where $\epsilon_k=\rho\norm{x_k-x_{k-1}}^2$, $d_k=\gamma\norm{x_k-x_{k-1}}$.

\textbf{Step 2:} set $z_{-1}=x_k$ and generate $z_0,z_1,\dots$ by $T_g(x_k,u_k,\lambda_k)$ until
\begin{equation}\tag{\ref{eq:decrease:g}}
g(x_k)-g(z_{i})-\innerp{u_k}{x_k-z_{i}}\geq \frac{(1-\sigma)}{\lambda}\norm{z_{i}-x_k}^2
\end{equation}
$\quad\quad\quad\;\;\;$and
\begin{equation}\label{eq:constraint:xik:IiPLDCA}
\dist\parentheses{u_k,\hat{\partial}_{\zeta_k}g(z_{i})}\leq \theta\norm{z_{i}-x_k}.
\end{equation}
$\quad\quad\quad\;\;\;$hold simultaneously for some $i\in\braces{-1}\cup\Natural$. Then set $x_{k+1}\coloneqq z_i$.

\textbf{Step 3:} $k\leftarrow k+1$.
\caption{A Feasible Implementation of tPLDCA}\label{alg:step:two}
\end{algorithm}

\subsection{Convergence Properties}
Since $\hat{\partial}_{\zeta_k}g(z_i)\subset \partial_{\zeta_k}g(z_i)$, it can be verified that (\ref{eq:constraint:xik:IiPLDCA}) is a stricter condition for $x_{k+1}$ than (\ref{eq:constraint:xik}). Therefore, Algorithm \ref{alg:step:two} can be viewed as a special implementation of tPLDCA and inherit all convergence properties from it. This statement is formalized as follows. 

\begin{theorem}\label{thm:convergence:IiPLDCA}
The sequence $(x_k)_{k\in\Natural}$ generated by Algorithm \ref{alg:step:two} satisfies:
\begin{enumerate}
\item[(a)] $\parentheses{f(x_k)+(\frac{\gamma}{2}+\rho)\norm{x_k-x_{k-1}}^2}_{k\in\Natural}$ is decreasing.

\item[(b)] $\lim_{k\to +\infty}\norm{x_{k+1}-x_k}=0$.

\item[(c)] $(f(x_k))_{k\in\Natural}$ is convergent.
\end{enumerate}
Moreover, if there exists a cluster point, say $\bar{x}$, of $(x_k)_{k\in\Natural}$, then:
\begin{enumerate}
\item[(d)] $\lim_{k\to +\infty} f(x_k)=f(\bar{x})$.

\item[(e)] $\bar{x}$ is a critical point of $f$, i.e., $\partial g(\bar{x})\cap \partial h(\bar{x})\neq \emptyset$.
\end{enumerate}
\end{theorem}
\begin{proof}
Since $\hat{\partial}_{\zeta_k} g(x_{k+1})\subset \partial_{\zeta_k} g(x_{k+1})$ from Proposition \ref{pp:subset:hatpartial}, the following holds
\begin{equation*}
(\forall k\in\Natural)\;\;\dist(u_k,\partial_{\zeta_k}g(x_{k+1}))\leq \dist(u_k,\hat{\partial}_{\zeta_k}g(x_{k+1}))\leq \theta\norm{x_{k+1}-x_k},
\end{equation*}
Therefore, for any $k\in\Natural$, $x_{k+1}$ satisfies (\ref{eq:constraint:xik}). Since $u_k$ satisfies (\ref{eq:constraint:uk}) and $x_{k+1}$ satisfies (\ref{eq:decrease:g}), we conclude that every sequence $(x_k)_{k\in\Natural}$ generated by Algorithm \ref{alg:step:two} can be viewed as a sequence generated by tPLDCA. Therefore, the conclusion follows from Theorem \ref{thm:convergence:functional} and \ref{thm:convergence:point}.
\end{proof}

\subsection{Finite Termination of Inner Loop}
\label{subsec:finite_termination_of_ItPLDCA}
Since the only difference between Algorithm \ref{alg:step:two} and tPLDCA lies in (\ref{eq:constraint:xik:IiPLDCA}), Proposition \ref{pp:decrease:g:finite} still holds for Algorithm \ref{alg:step:two}. Therefore, to prove the finite inner loop termination of Algorithm \ref{alg:step:two}, we only need to reestablish Proposition \ref{pp:constraint:xi:finite} for (\ref{eq:constraint:xik:IiPLDCA}). See the following proposition.

\begin{proposition}(Achievability of (\ref{eq:constraint:xik:IiPLDCA}))\label{pp:constraint:xi:finite:IiPLDCA}
Suppose that $x_k$ is not a solution of (PP), then (\ref{eq:constraint:xik:IiPLDCA}) remains satisfied after finite iterations of $T_g(x_k,u_k,\lambda)$, i.e., $\exists N_k\in \Natural$ such that
\begin{equation*}
(\forall i\geq N_k)\;\;\dist(u_k,\hat{\partial}_{\zeta_k}g(z_{i}))\leq \theta\norm{z_{i}-x_k}.
\end{equation*}
\end{proposition}
\begin{proof}
From Assumption \ref{assumpt:solution:steptwo}, $0\not\in\tilde{f}_k(x_k)$ and $\lambda_k=\lambda$, there exists $\bar{z}_{k}\neq x_k$ such that $u_k-\frac{1}{\lambda}(\bar{z}_{k}-x_k)\in\partial g(\bar{z}_{k})$.
Consider the following function:
\begin{equation*}
D(z)\coloneqq \dist\parentheses{u_k,\hat{\partial}_{\zeta_k}g(z)}-\theta\norm{z-x_k}.
\end{equation*}
Since $u_k-\frac{1}{\lambda}(\bar{z}_{k}-x_k)\in\partial g(\bar{z}_{k})\subset\hat{\partial}_{\zeta_k}g(\bar{z}_k)$, we have
\[\dist\parentheses{u_k,\hat{\partial}_{\zeta_k}g(\bar{z}_{k})}\leq \frac{1}{\lambda}\norm{\bar{z}_{k}-x_k}\]
thus
\begin{equation*}
D(\bar{z}_{k})\leq \parentheses{\frac{1}{\lambda}-\theta}\norm{\bar{z}_{k}-x_k}<0.
\end{equation*}
Since $z_i\to\bar{z}_{k}$, the upper semicontinuity of $D(z)$ (Theorem \ref{thm:upper:semicont:estrict}) then yields
\begin{equation*}
\limsup_{i\to+\infty}D(z_i)\leq D(\bar{z}_{k})<0.
\end{equation*}
Therefore, there exists $N_k\in\Natural$ such that
\begin{equation*}
(\forall i\geq N_k)\;\; D(z_i)\leq 0.
\end{equation*}
Expanding the expression of $D(z_i)$ yields
\begin{equation*}
(\forall i\geq N_k)\;\; \dist\parentheses{u_k,\hat{\partial}_{\zeta_k}g(z_i)}\leq\theta\norm{z_i-x_k},
\end{equation*}
which completes the proof.
\end{proof}

Applying Proposition \ref{pp:constraint:xi:finite:IiPLDCA}, we can prove the finite inner loop termination of Algorithm \ref{alg:step:two}.

\begin{theorem}[Inner loop termination of Algorithm \ref{alg:step:two}]\label{thm:inner_loop_termination_of_Alg3}
The step 2 of Algorithm \ref{alg:step:two} terminates within finite inner iterations, i.e., there exists $N_k\in \Natural$ such that $z_{N_k}$ satisfies simultaneously
\begin{equation}\tag{\ref{eq:decrease:g:finite}}
g(x_k)-g(z_{N_k})-\innerp{u_k}{x_k-z_{N_k}}\geq \frac{(1-\sigma)}{\lambda}\norm{z_{N_k}-x_k}^2
\end{equation}
and
\begin{equation}\label{eq:constraint:xi:finite:IiPLDCA}
\dist(u_k,\hat{\partial}_{\zeta_k}g(z_{N_k}))\leq \theta\norm{z_{N_k}-x_k}.
\end{equation}
\end{theorem}
\begin{proof}
From Assumption \ref{assumpt:solution:steptwo} and $\lambda_k=\lambda$, there exists $\bar{z}_{k}\in\dom{g}$ such that 
\[u_k-\frac{1}{\lambda}(\bar{z}_{k}-x_k)\in\partial g(\bar{z}_{k}).\]
\begin{enumerate}
\item[(I)] If $\bar{z}_{k}=x_k$, then since $z_0=x_k$ (Assumption \ref{assumpt:solution:steptwo}(a)), we have
\begin{equation*}
g(x_k)-g(z_0)-\innerp{u_k}{x_k-z_0}=0\geq \frac{(1-\sigma)}{\lambda}\norm{z_0-x_k}^2=0.
\end{equation*}
In addition, since $z_0=x_k=\bar{z}_{k}$, we have $u_k\in \partial g(z_0)\subset\hat{\partial}_{\zeta_k}g(z_0)$, thus
\begin{equation*}
\dist\parentheses{u_k,\hat{\partial}_{\zeta_k}g(z_0)}=0\leq \theta\norm{z_0-x_k}^2=0.
\end{equation*}
Therefore, $N_k=0$ satisfies (\ref{eq:decrease:g:finite}) and (\ref{eq:constraint:xi:finite:IiPLDCA}).

\item[(II)] If $\bar{z}_{k}\neq x_k$. Then from Proposition \ref{pp:decrease:g:finite}, there exists $N_{1,k}\in\Natural$ such that
\begin{equation*}
(\forall i\geq N_{1,k})\;\; g(x_k)-g(z_i)-\innerp{u_k}{x_k-z_i}\geq \frac{(1-\sigma)}{\lambda}\norm{z_i-x_k}.
\end{equation*}
From Proposition \ref{pp:constraint:xi:finite:IiPLDCA}, there exists $N_{2,k}\in\Natural$ such that
\begin{equation*}
(\forall i\geq N_{2,k})\;\; \dist(u_k,\hat{\partial}_{\zeta_k}g(z_i))\leq \theta \norm{z_i-x_k}.
\end{equation*}
Therefore, $N_k\coloneqq \max\{N_{1,k},N_{2,k}\}$ satisfies (\ref{eq:decrease:g:finite}) and (\ref{eq:constraint:xi:finite:IiPLDCA}).
\end{enumerate}
Combining (I) and (II) completes the proof.
\end{proof}

\section{Numerical Results}\label{sec:numerical_results}

In this section, we present a simple numerical example to illustrate how to implement Algorithm \ref{alg:step:two}, and to verify its convergence properties and finite inner loop termination. 

We consider $f(x)\coloneqq g(x)-h(x)$ with $x\coloneqq \parentheses{x_{a},x_{b}}\in\Real^2$ and $g,h$ defined as follows:
\begin{align*}
g(x) &\coloneqq x_{a}^2+x_{b}^2+x_{a}x_{b}+\max(-x_a,0),\\
h(x) &\coloneqq \frac{1}{2}(x_b-1)^2.
\end{align*}
Then to minimize $f$ by Algorithm \ref{alg:step:two}, the following requirements must be satisfied:
\begin{enumerate}
\item[(a)] For every $k\in\Natural$, we can obtain $u_k\in\Real^2$ satisfying (\ref{eq:constraint:uk}).

\item[(b)] There exists an inner iterative algorithm $T_g$ satisfying Assumption \ref{assumpt:solution:steptwo}.

\item[(c)] For every $u_k,x_k\in\Real^2$, the conditions (\ref{eq:decrease:g}) and (\ref{eq:constraint:xik:IiPLDCA}) can be examined exactly.
\end{enumerate}
In the sequel, we show in turn that the three requirements above are all satisfied in the considered DC program. 

First, since $h$ is differentiable with $\nabla h(x)=(0,x_b-1)$, setting $u_k=\nabla h(x_k)$ satisfies (\ref{eq:constraint:uk}) regardless of the values of $\rho$ and $\gamma$. 

Second, according to Assumption \ref{assumpt:solution:steptwo}, the inner iterative algorithm $T_g(x_k,u_k,\lambda_k)$ generates a sequence $(z_i)_{i\in\Natural}$ which converges to a minimizer of 
\begin{align*}
\tilde{f}_k(z) \coloneqq g(z)-\innerp{u_k}{z-x_k}+\frac{1}{2\lambda_k}\norm{z-x_k}^2.
\end{align*}
One can verify that $\tilde{f}_k$ can be decomposed as the sum of a differentiable convex function with Lipschitz continuous gradient and a proximable convex function (i.e., the proximity operator \cite[Sec. 12.4]{bauschke2017} of which can be computed efficiently), hence we can adopt ISTA (see \cite[Sec. 2.3]{beck2009} and \cite{combettes2005}) as the inner iterative algorithm\footnote{One may notice that for this simple 2D example, the minimizer of $\tilde{f}_k$ has a closed-form expression. Despite this, we minimize $\tilde{f}_k$ by an iterative algorithm to showcase the finite inner loop termination of Algorithm \ref{alg:step:two}.}. The inner iterate sequence $(z_i)_{i\in\Natural}$ can be computed iteratively as follows:
\begin{equation*}
z_{i}\coloneqq \textrm{ISTA}(z_{i-1};k),
\end{equation*}
where $\textrm{ISTA}(\cdot;k)$ is the ISTA iteration step \cite[Sec. 2.3]{beck2009} for minimizing $\tilde{f}_k$.

Third, for given $x_k,u_k\in\Real^2$, the condition (\ref{eq:decrease:g}) is equivalent to
\begin{equation}
\Delta_{k,i}^{(\ref{eq:decrease:g})}\coloneqq g(x_k)-g(z_{i})-\innerp{u_k}{x_k-z_{i}}- \frac{(1-\sigma)}{\lambda}\norm{z_{i}-x_k}^2\geq 0.
\end{equation}
Since $\Delta_{k,i}^{(\ref{eq:decrease:g})}$ can be computed in closed-form, (\ref{eq:decrease:g}) can be examined exactly. On the other hand, (\ref{eq:constraint:xik:IiPLDCA}) is equivalent to 
\begin{equation}
\Delta_{k,i}^{(\ref{eq:constraint:xik:IiPLDCA})} \coloneqq \theta\norm{z_i-x_k}-\dist\parentheses{u_k,\hat{\partial}_{\zeta_k}g(z_i)}\geq 0.
\end{equation}
Note that $g(x)=\max(g_1(x),g_2(x))$, where
\begin{align*}
g_1(x) &\coloneqq x_a^2+x_b^2+x_ax_b-x_a,\\
g_2(x) &\coloneqq x_a^2+x_b^2+x_ax_b,
\end{align*}
we can obtain the closed-form expression\footnote{We note that for this example, $\partial g(z_i)\subsetneq \hat{\partial}_{\zeta_k}g(z_i)\subsetneq \partial_{\zeta_k}g(z_i)$, thus Algorithm \ref{alg:step:two} is a nontrivial extension of Algorithm \ref{alg:souza}, and is a non-straightforward implementation of Algorithm \ref{alg:concept}.} of $\hat{\partial}_{\zeta_k}g(z_i)$ by Definition \ref{def:estrict:subdiff}:
\begin{equation*}
\hat{\partial}_{\zeta_k}g(z_i)=\begin{cases}
\braces{(2z_{i,a}+z_{i,b}-1,2z_{i,b}+z_{i,a})}, & \textrm{if }z_{i,a}<-\zeta_k;\\
[2z_{i,a}+z_{i,b}-1,2z_{i,a}+z_{i,b}]\times\braces{2z_{i,b}+z_{i,a}}, & \textrm{if }-\zeta_k\leq z_{i,a}\leq \zeta_k;\\
\braces{(2z_{i,a}+z_{i,b},2z_{i,b}+z_{i,a})}, &\textrm{if }z_{i,a}>\zeta_k.
\end{cases}
\end{equation*}
Accordingly, $\Delta_{k,i}^{(\ref{eq:constraint:xik:IiPLDCA})}$ can be computed in closed-form and (\ref{eq:constraint:xik:IiPLDCA}) can be examined exactly.

\begin{figure}[h]
\centering
\begin{tabular}{c}
\begin{minipage}{0.48\linewidth}
\centering
\includegraphics[height=4cm]{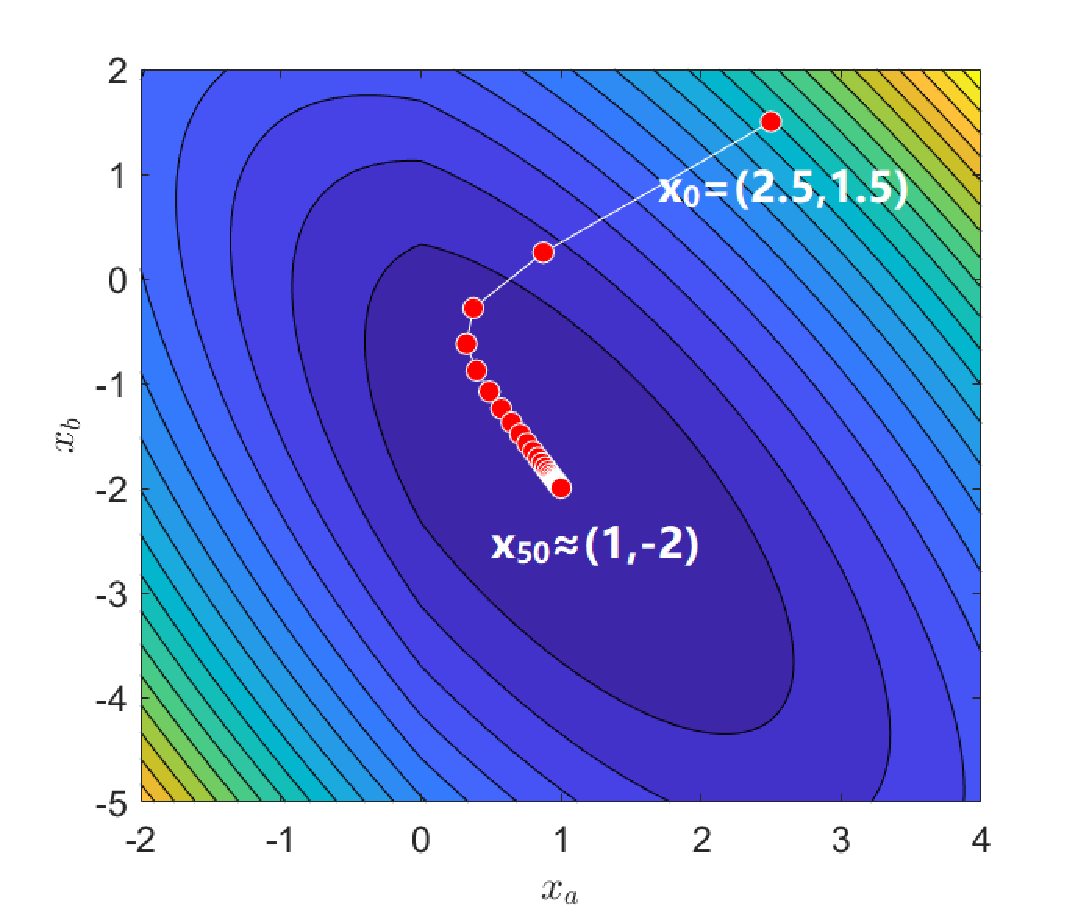}\caption{The contours of $f$ and $(x_k)_{k=0}^{50}$.}\label{fig:contours_and_iterates}
\end{minipage}
\hfill
\begin{minipage}{0.48\linewidth}
\centering
\includegraphics[height=4cm]{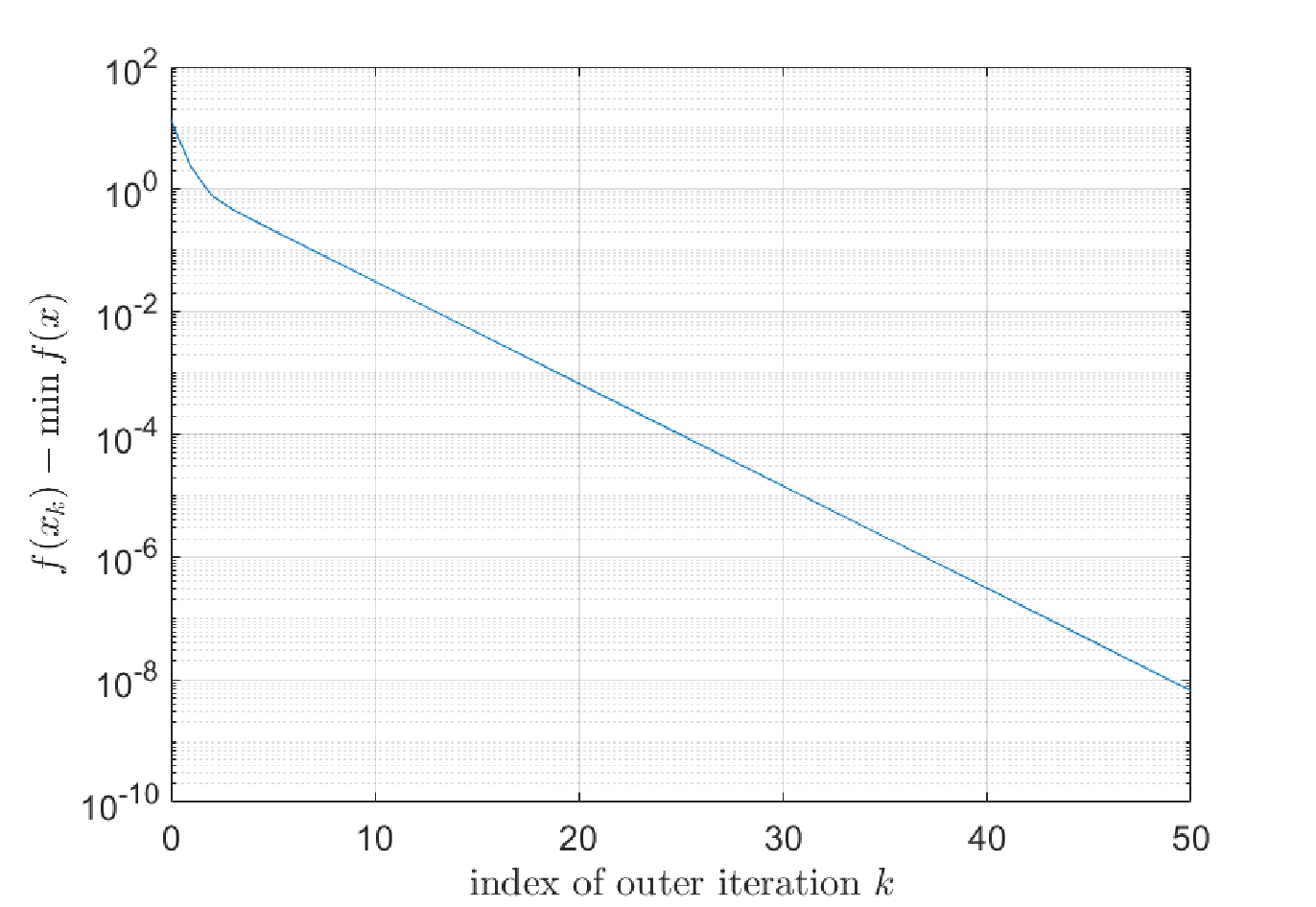}\caption{The residual error ${f(x_k)-\min_x f(x)}$.}\label{fig:residual_error}
\end{minipage}
\end{tabular}
\end{figure}

Applying the results above to Algorithm \ref{alg:step:two} yields the pseudocode for solving the considered DC program; See Algorithm \ref{alg:pseudocode}. We set the hyperparameters $\sigma=0.01,\lambda=1,\theta=1.1/\lambda,x_0=(2.5,1.5)$ and $\zeta_k=1/k^2$, and run Algorithm \ref{alg:pseudocode} for 50 outer iterations.

The results are as follows. Fig. \ref{fig:contours_and_iterates} shows the contours of the objective function $f$, and the positions of the iterates $(x_k)_{k=0}^{50}$. Fig. \ref{fig:residual_error} shows the residual error $f(x_k)-\min_{x}f(x)$ versus the index of outer iteration $k$. From Fig. \ref{fig:contours_and_iterates} and \ref{fig:residual_error}, one can verify that Algorithm \ref{alg:step:two} converges to a critical point of $f$, namely, the global minimizer of $f$ in this problem. 

Fig. \ref{fig:del_a} and \ref{fig:del_b} show the dependency of the inner loop termination metrics $\Delta_{k,i}^{(\ref{eq:decrease:g})}$ and $\Delta_{k,i}^{(\ref{eq:constraint:xik:IiPLDCA})}$ on the index of inner iteration $i$ with different values of $k$. The two figures demonstrate that as Algorithm \ref{alg:step:two} proceeds, (\ref{eq:decrease:g}) and (\ref{eq:constraint:xik:IiPLDCA}) can always get satisfied simultaneously within finite iterations, which verifies the finite inner loop termination property we claimed in Sec. \ref{subsec:finite:terminate:CiPLDCA} and \ref{subsec:finite_termination_of_ItPLDCA}.

\begin{figure}
\centering
\subfloat[][$k=10$.]{\includegraphics[width=5cm]{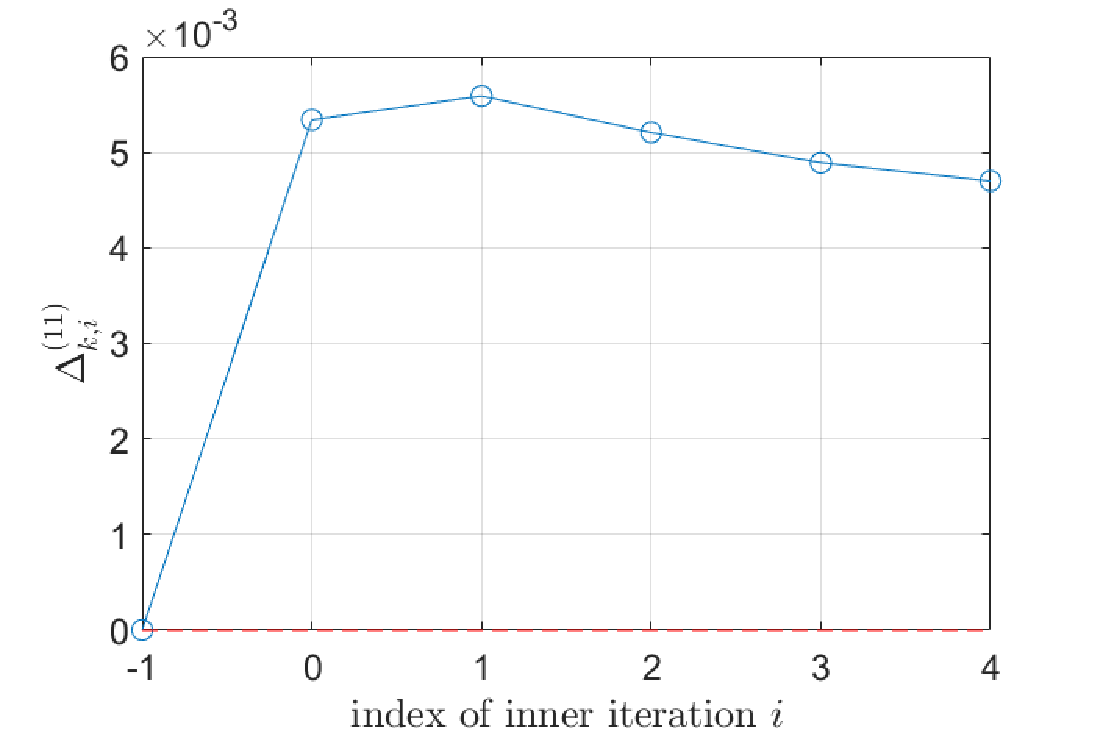}}
\subfloat[][$k=30$.]{\includegraphics[width=5cm]{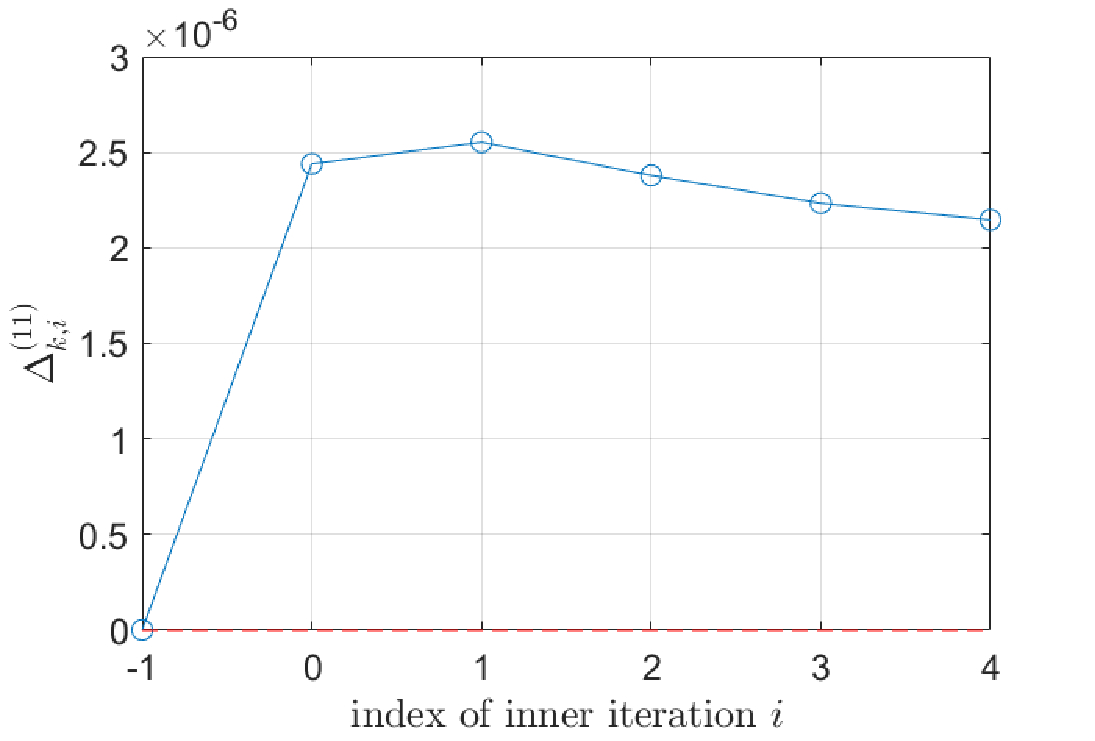}}
\subfloat[][$k=50$.]{\includegraphics[width=5cm]{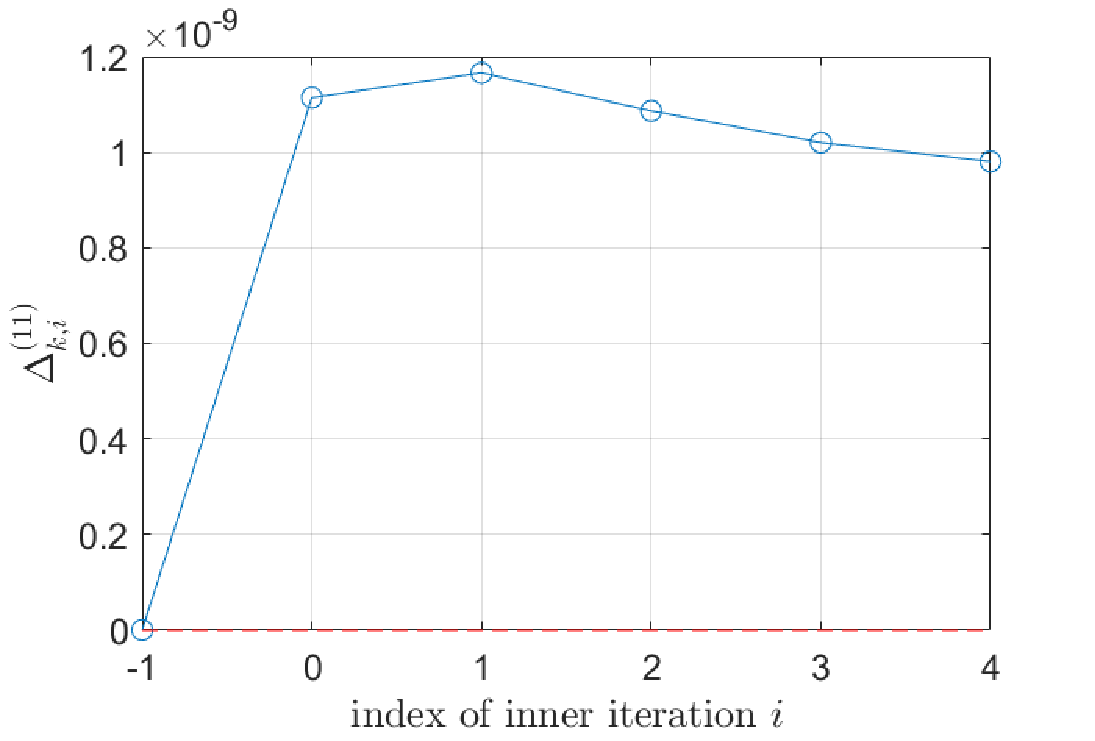}}
\caption{$\Delta_{k,i}^{(\ref{eq:decrease:g})}$ versus the index of inner iteration $i$.}
\label{fig:del_a}
\end{figure}
\begin{figure}
\centering
\subfloat[][$k=10$.]{\includegraphics[width=5cm]{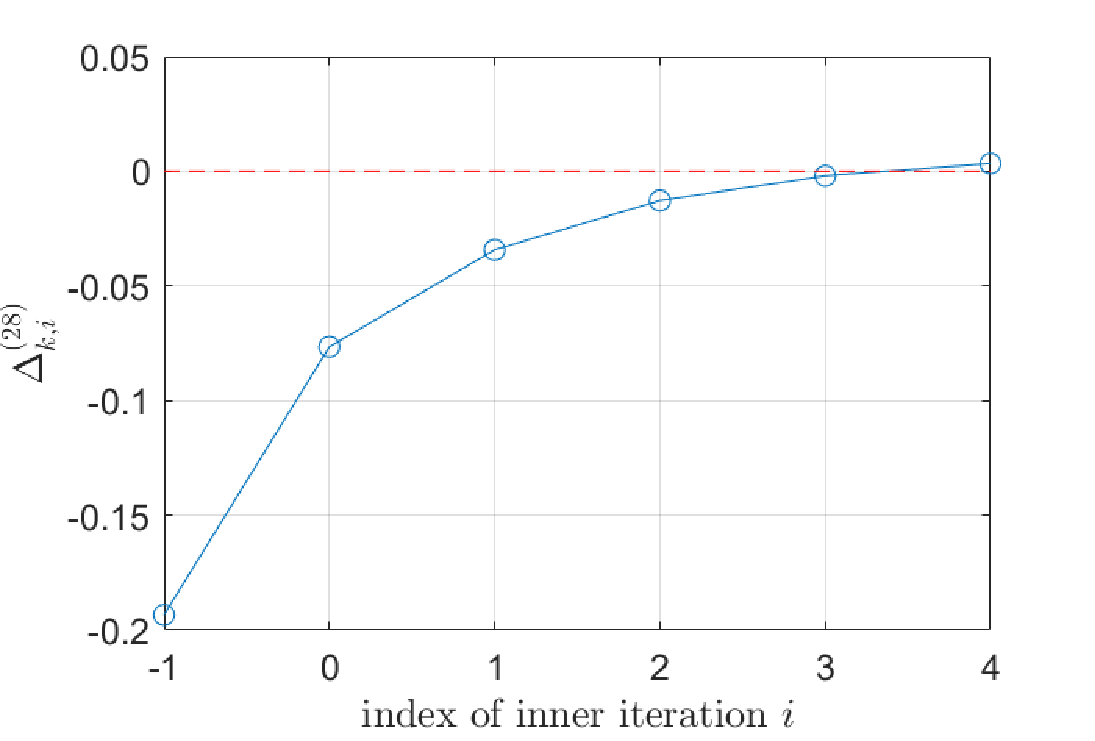}}
\subfloat[][$k=30$.]{\includegraphics[width=5cm]{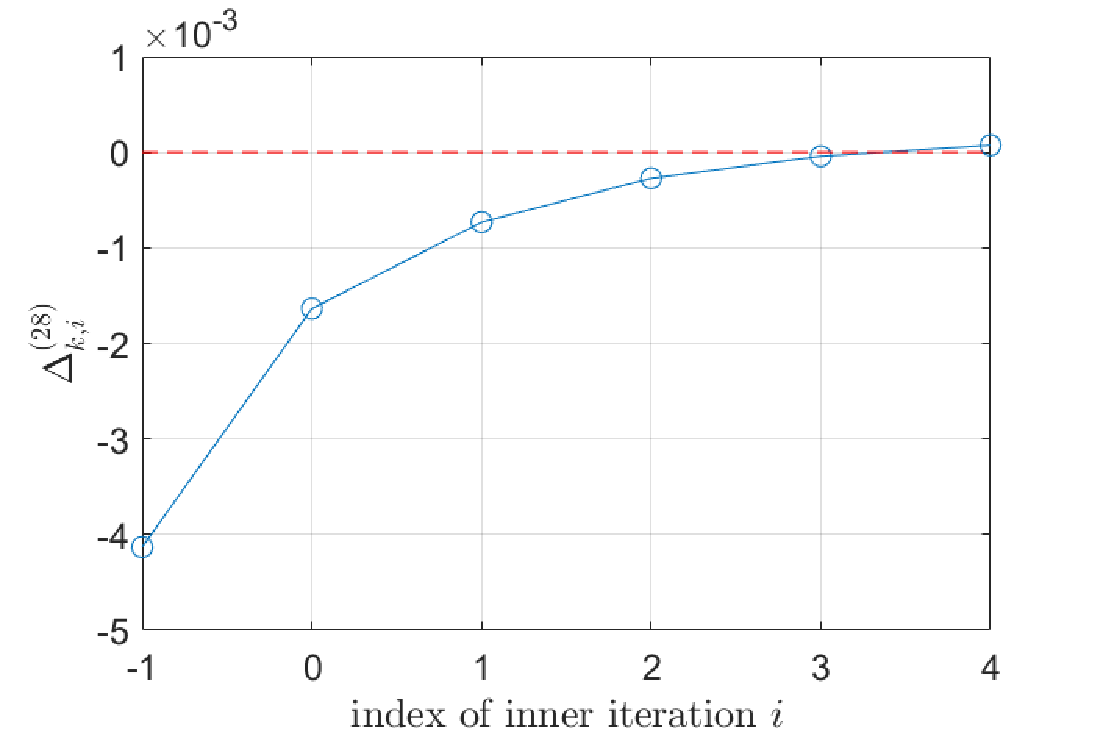}}
\subfloat[][$k=50$.]{\includegraphics[width=5cm]{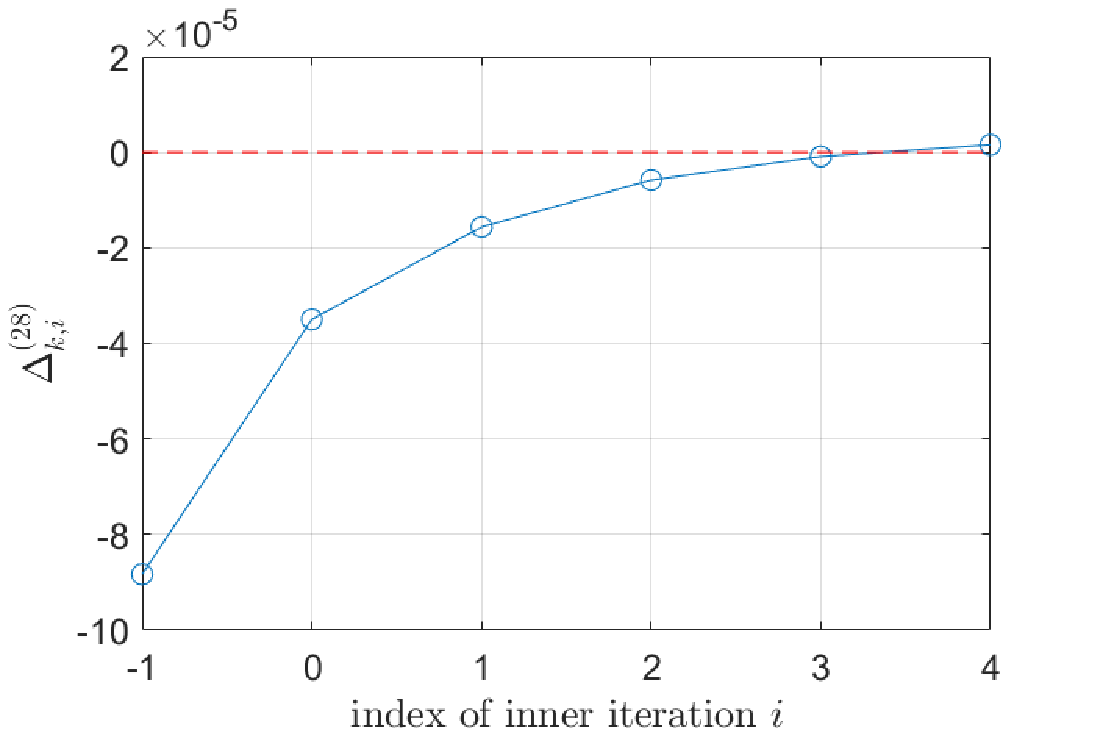}}
\caption{$\Delta_{k,i}^{(\ref{eq:constraint:xik:IiPLDCA})}$ versus the index of inner iteration $i$.}
\label{fig:del_b}
\end{figure}

\begin{algorithm}

\textbf{Input:} $\sigma\in(0,1),\lambda>0,\theta>\frac{1}{\lambda},x_0\in\Real^2,(\zeta_k)_{k\in\Natural}\subset\Real_{++}$ satisfying $\lim_{k\to +\infty} \zeta_k=0$.

\textbf{Output:} a sequence $(x_k)_{k\in\Natural}$ generated by Algorithm \ref{alg:step:two}.

\begin{algorithmic}
\State $k\gets 0$
\State $x_k\gets x_0$
\Repeat
\State $u_k\gets \nabla h(x_k)$
\State $i\gets -1$
\State $z_i\gets x_k$
\While{$\Delta_{k,i}^{(\ref{eq:decrease:g})}<0$ or $\Delta_{k,i}^{(\ref{eq:constraint:xik:IiPLDCA})}<0$}
\State $i\gets i+1$
\State $z_i\gets \textrm{ISTA}(z_{i-1};k)$
\EndWhile
\State $k\gets k+1$
\State $x_k\gets z_i$
\Until{$x_k=x_{k-1}$}
\end{algorithmic}

\caption{The Pseudocode for Solving the Considered DC Program}\label{alg:pseudocode}
\end{algorithm}

\section{Conclusion}\label{sec:conclusion}
In this paper, we have pointed out the termination issue of existing inexact DC algorithms by presenting counterexamples. To resolve this termination issue and to permit greater inexactness in computation, we have proposed a novel inexact proximal linearized DC algorithm termed tPLDCA. We have established convergence guarantees of the sequence generated by tPLDCA. Moreover, we have proved that for every inner iterative algorithm that converges to a solution of the proximal point subproblem, the inner loop of tPLDCA is guaranteed to terminate in finite iterations. In addition, a feasible implementation of tPLDCA has been proposed under the assumption that the first convex component of the DC function is the pointwise maximum of finitely many convex smooth functions. Numerical results have demonstrate the effectiveness of the proposed implementation of tPLDCA.

\section*{Funding}
This work was supported in part by JSPS Grants-in-Aid (21J22393) and by JSPS Grants-in-Aid (19H04134).

\bibliographystyle{tfnlm}
\bibliography{ref}

\end{document}